\documentclass[smallextended]{svjour3}
\usepackage{xcolor}
\usepackage{enumerate}
\usepackage{amsmath}
\usepackage{amssymb}

\usepackage{multirow}
\usepackage{mathrsfs}
\usepackage[bmargin=1.35in,rmargin=1.7in,asymmetric]{geometry}
\usepackage{booktabs}
\usepackage{mathtools}
\usepackage{color}
\usepackage{algorithm}
\usepackage{algorithmicx}
\usepackage{url}
\usepackage{latexsym}
\usepackage{booktabs}
\usepackage{bm}
\usepackage{graphicx}
\usepackage[flushleft]{threeparttable}
\usepackage{algpseudocode}
\usepackage{textcomp}
\usepackage{subfigure}
\usepackage[toc,page]{appendix}
\usepackage{textcomp}
\newcommand{\nfR}{\mathbb{R}}

\newcommand{\lip}{\left<}
\newcommand{\rip}{\right>}
\newcommand{\lvt}{\left[}
\newcommand{\rvt}{\right]}
\newcommand{\idl}{\mbox{Ideal}}
\newcommand{\qmod}{\mbox{Qmod}}
\newcommand{\ddd}{,\ldots,}

\newcommand{\fall}{\quad \mbox{for all}\quad}
\newcommand{\vi}{\mbox{VI}}
\newcommand{\VIXF}{\mbox{VI}\left(X,\mbF\right)}

\DeclareMathOperator{\Rank}{rank}

\newcommand{\wt}{\widetilde}
\newcommand{\td}{\tilde}
\newcommand{\wtU}{\widetilde{\mathcal{U}}}
\newcommand{\wtV}{\widetilde{\mathcal{V}}}
\newcommand{\wtW}{\widetilde{\mathcal{W}}}
\newcommand{\re}{\mathbb{R}}
\newcommand{\cpx}{\mathbb{C}}

\newcommand{\N}{\mathbb{N}}

\newcommand{\Pj}{\mathbb{P}}

\def\af{\alpha}

\def\gm{\gamma}

\newcommand{\st}{\mathit{s.t.}}
\newcommand{\reff}[1]{(\ref{#1})}

\newcommand{\lmd}{\lambda}
\newcommand{\pt}{\partial}
\newcommand{\dt}{\delta}

\newcommand{\mbF}{F}
\newcommand{\mc}[1]{\mathcal{#1}}

\newcommand{\defeq}{\,:=\,}
\def\rank{\mbox{rank}}
\def\cod{\mbox{codim}\,}
\newcommand{\bdes}{\begin{description}}
\newcommand{\edes}{\end{description}}

\newcommand{\SOLXF}{\mbox{SOL}\left(X,\mbF\right)}

\newcommand{\bal}{\begin{align}}
\newcommand{\eal}{\end{align}}
\newcommand{\htm}{\hat{m}}
\newcommand{\tdx}{\td{x}}

\newcommand{\bnum}{\begin{enumerate}}
\newcommand{\enum}{\end{enumerate}}

\newcommand{\bit}{\begin{itemize}}
\newcommand{\eit}{\end{itemize}}

\newcommand{\bea}{\begin{eqnarray}}
\newcommand{\eea}{\end{eqnarray}}
\newcommand{\be}{\begin{equation}}
\newcommand{\ee}{\end{equation}}

\newcommand{\baray}{\begin{array}}
\newcommand{\earay}{\end{array}}

\newcommand{\bsry}{\begin{subarray}}
\newcommand{\esry}{\end{subarray}}

\newcommand{\bca}{\begin{cases}}
\newcommand{\eca}{\end{cases}}

\newcommand{\bcen}{\begin{center}}
\newcommand{\ecen}{\end{center}}

\newcommand{\bbm}{\begin{bmatrix}}
\newcommand{\ebm}{\end{bmatrix}}

\newcommand{\btab}{\begin{tabular}}
\newcommand{\etab}{\end{tabular}}

\spnewtheorem{thm}{Theorem}[section]{\bf}{\it}
\spnewtheorem{prop}[thm]{Proposition}{\bf}{\it}
\spnewtheorem{lem}[thm]{Lemma}{\bf}{\it}
\spnewtheorem{cor}[thm]{Corollary}{\bf}{\it}
\spnewtheorem{ass}[thm]{Assumption}{\bf}{\it}
\spnewtheorem{defi}[thm]{Definition}{\bf}{\it}
\spnewtheorem{alg}[thm]{Algorithm}{\bf}{\it}
\spnewtheorem{exm}[thm]{Example}{\bf}{\rm}
\spnewtheorem{rmk}[thm]{Remark}{\bf}{\rm}
\spnewtheorem{cond}[thm]{Condition}{\bf}{\it}
\spnewtheorem{quest}[thm]{Question}{\bf}{\it}

\numberwithin{equation}{section}
\journalname{}
\begin{document}

\title{Solving polynomial variational inequality problems via Lagrange multiplier expressions and Moment-SOS relaxations}
\titlerunning{Polynomial VIPs}
% \date{}
\author{Jiawang Nie \and Defeng Sun \and Xindong Tang \and Min Zhang}
\authorrunning{J.~Nie, D.~Sun, X.~Tang and M.~Zhang}

\institute{
Jiawang~Nie \at
Department of Mathematics, University of California San Diego,
9500 Gilman Drive, La Jolla, CA, USA, 92093. \\
\email{njw@math.ucsd.edu}
\and
Defeng Sun \at
  Department of Applied Mathematics,
The Hong Kong Polytechnic University,
Hung Hom, Kowloon, Hong Kong. \\
\email{defeng.sun@polyu.edu.hk}
\and
Xindong Tang \at
Department of Applied Mathematics,
The Hong Kong Polytechnic University,
Hung Hom, Kowloon, Hong Kong. \\
\email{xindong.tang@polyu.edu.hk}
\and
Min Zhang \at
Department of Applied Mathematics,
The Hong Kong Polytechnic University,
Hung Hom, Kowloon, Hong Kong. \\
\email{min-opt.zhang@polyu.edu.hk}
}
% \email{min-opt.zhang@polyu.edu.hk}
% \author*[1]{\fnm{Jiawang} \sur{Nie}}\email{njw@math.ucsd.edu}
% % \equalcont{These authors contributed equally to this work.}

% \author[2]{\fnm{Defeng} \sur{Sun}}\email{defeng.sun@polyu.edu.hk}
% % \equalcont{These authors contributed equally to this work.}

% \author[2]{\fnm{Xindong} \sur{Tang}}\email{xindong.tang@polyu.edu.hk}
% % \equalcont{These authors contributed equally to this work.}

% \author[2]{\fnm{Min} \sur{Zhang}}\email{min-opt.zhang@polyu.edu.hk}
% % \equalcont{These authors contributed equally to this work.}

% \affil*[1]{\orgdiv{Department of Mathematics}, \orgname{University of California}, \orgaddress{\street{9500 Gilman Drive}, \city{La Jolla}, \postcode{92093}, \state{CA}, \country{USA}}}

% \affil[2]{\orgdiv{Department of Applied Mathematics}, \orgname{The Hong Kong Polytechnic University}, \orgaddress{\street{Hung Hom}, \city{Kowloon}, \country{Hong Kong}}}
\maketitle

\begin{abstract}
    In this paper, we study variational inequality problems (VIPs) with involved mappings and feasible sets characterized by polynomial functions (namely, polynomial VIPs).
    We propose a numerical algorithm for computing solutions to polynomial VIPs based on Lagrange multiplier expressions and the Moment-SOS hierarchy of semidefinite relaxations.
    We also extend our approach to finding more or even all solutions to polynomial VIPs.
    We show that the method proposed in this paper can find solutions or detect the nonexistence of solutions within finitely many steps, under some general assumptions.
    In addition, we show that if the VIP is given by generic polynomials,
    then it has finitely many Karush-Kuhn-Tucker points,
    and our method can solve it within finitely many steps.
    Numerical experiments are conducted to illustrate the efficiency of the proposed methods.
\end{abstract}



\section{Introduction}
Let $\re^n$ be the $n$-dimensional Euclidean space and $X\subseteq \re^n$ be a nonempty closed set.
For a tuple of functions $F(x):=(F_1(x)\ddd F_n(x))$ where every $F_i: X \rightarrow \mathbb{R}$,
the {\it variational inequality problem} (VIP) is
\be \label{eq:VI}
\mbox{Finding}\quad x\in X \quad \mbox{such that} \quad
(y-x)^T F(x)\ge0 \quad \mbox{for all} \quad y\in X. \ee
Denote the problem (\ref{eq:VI}) by $\VIXF$
and its solution set by $\SOLXF$.
We call the $\VIXF$ a \emph{polynomial VIP} if $F_1\ddd F_n$ are polynomial functions in $x$
and $X$ is characterized by
\be\label{eq:X}
X=\{x\in\re^n: g_i(x)\ge0 \, (i\in\mc{I}),\ g_i(x)=0 \, (i\in\mc{E})\},\ee
where $\mc{E}$ and $\mc{I}$ are two disjoint index sets and every $g_i\, (i\in\mc{E}\cup\mc{I})$ is a polynomial in $x$.

Variational inequalities were first proposed in the 1960s and have been in the central position of the optimization field since then.
As a uniform approach to handling optimization and equilibrium problems, variational inequalities have various applications in decision-making, economy, finance, engineering, and other fields.
We refer to the monograph \cite{Facchinei2007book} for a general introduction to VIPs.

Methods for solving VIPs have been extensively investigated in the literature (see \cite{Facchinei2007book,k2001} and references therein). Most approaches are designed for handling the $\VIXF$ under the typical assumption that the set $X$ is convex.
For such cases, it is well known that the $\VIXF$ can be equivalently reformulated into the following equation:
\begin{eqnarray}\label{nat}
F^{\textrm{nat}}(x) := x - \Pi_X(x-F(x))=0, \quad
\end{eqnarray}
where $\Pi_X(\cdot)$ means the projection onto $X$. This reformulation leads to the popular projection-type methods, such as the basic fixed-point method \cite{s1970}, the extragradient method \cite{k1977}, and the hyperplane projection method \cite{k1997}.
Usually, the pseudomonotonicity of map $F$ is required to guarantee the convergence of the projection-type methods.
% even requires $\textbf{F}$ to be Lipschitz continuous.Let $x=\Pi_X(z)$ in \reff{nat}.
Besides \eqref{nat}, another equivalent equation reformulation of the $\VIXF$ is as follows \cite{r1992}: %Besides that, with a change of variable, solving the $\VIXF$ is equivalent to handling the following equation \cite{r1992}:
\be\label{nor}
F^{\textrm{nor}}(z) := F(\Pi_X(z)) + z - \Pi_X(z)=0.
\ee
With a change of variable, if $x^*\in \textrm{SOL}(X, F)$, then $z^* := x^* - F(x^*)$ is a solution to \eqref{nor}, and conversely, if $z^*\in X$ solves \eqref{nor}, then $x^*:=\Pi_X(z^*)$ is a solution to the $\VIXF$.
Numerical approaches such as the nonsmooth Newton method \cite{p1990,r1994}, the interior point method \cite{kmny1991,y1997}, and the smoothing Newton method \cite{cqs1998,qsz2000}, were developed for solving the $\VIXF$ via dealing with the equivalent nonsmooth equation \eqref{nat} or \eqref{nor}.

However, when the set $X$ is nonconvex, the projection map $\Pi_X$ may not be single-valued (e.g., consider the projection from the origin to a ring $\{x\in \mathbb{R}^n: 1\leq \|x\|^2\leq 2\}$), which leads the nonsmooth equations \reff{nat} and \reff{nor} to become general inclusion problems.
To the best of our knowledge, no well-developed algorithm is applicable to solve general inclusion problems efficiently, especially when $F$ lacks the monotonicity-like property. Fortunately, if the set $X$ is finitely representable as a system of equalities and inequalities, then the $\VIXF$ associates with the Karush-Kuhn-Tucker (KKT) system (see details in Section~\ref{sc:pre}) under certain constraint qualification conditions (CQs).
The KKT system can be viewed as a mixed complementarity problem (MiCP), thus methods for solving MiCPs can be applied; see \cite{df1995,Facchinei2007book,nz2016}.
% However, there exists another gap for handling $\VIXF$ with nonconvex $X$ using the KKT conditions.
Every solution of the KKT system solves the $\VIXF$ when $X$ is convex.
However, this is not true under the nonconvex setting.
Indeed, when $X$ is nonconvex, solutions to the KKT system only provide candidate solutions to the $\VIXF$.
%Moreover, it is difficult to verify whether a KKT point is a solution to the $\VIXF$ when $X$ is nonconvex.
Furthermore,  when $X$ is bounded and convex, the existence of solutions to the $\VIXF$ is guaranteed. However, without boundedness or convexity of $X$,
the $\SOLXF$ may be empty,
and there is no well-known result for the existence of solutions.
In fact, when $X$ is nonconvex or unbounded,
how to detect the nonexistence of solutions is still open,
to the best of the authors' knowledge.

Recently, some studies on the theoretical properties of polynomial VIPs and their variations have been developed. Hieu \cite{Hieu2020} studied the solution set of the polynomial VIP with a convex feasible set.
Tensor complementarity problems and tensor variational inequality problems were considered in \cite{Bai2016,Wang2018}. Further, weakly homogeneous variational inequality problems, as extensions of polynomial VIPs and tensor VIPs, were studied in \cite{Gowda2019,Ma2020}.
Moreover, the Moment-SOS hierarchy of semidefinite relaxations was developed for solving polynomial optimization problems. It can find global minimizers for nonconvex polynomial optimization problems with theoretical guarantees under mild conditions.
It was also applied to solve many optimization-related problems given by polynomials,
such as bilevel optimization \cite{Jeyakumar2016,nie2017bilevel,nie2021bilevel},
saddle point problems and generalized Nash equilibrium problems (GNEPs) \cite{Nie2020nash,Nie2021convex,Nie2018saddle},
tensor (eigenvalue) complementarity problems \cite{Fan2018,Zhao2019}, etc.
We refer to \cite{lasserre2015introduction,LasICM,Lau09,LauICM,NieLoc}
for more references about polynomial optimization.

In this paper, we focus on numerical methods for solving polynomial VIPs without assuming the monotonicity of the map $\mbF$ or the convexity of the set $X$.
Our main contributions are listed below.
\begin{itemize}
    \item We propose an algorithm for solving polynomial VIPs using Lagrange multiplier expressions and the Moment-SOS relaxations. Under some general assumptions, the proposed algorithm can find a solution to the VIP or detect the nonexistence of solutions in finitely many steps. Particularly, this computational goal can be achieved in the initial loop when $X$ is convex.

    \item Based on the proposed algorithm for finding a solution to a polynomial VIP,
    we further investigate how to find more (even all, if there are finitely many) solutions.

    \item We show that if the VIP is given by generic polynomial functions,
    then there are finitely many KKT points, and all solutions to the polynomial VIP are KKT points.
    In this case, our method is guaranteed to terminate within finitely many loops.
    Moreover, we give an upper bound for the number of KKT points using the intersection theory.

    \item Even when the algorithms do not have finite convergence,
    we show the asymptotic convergence under certain continuity assumptions.
    Numerical experiments are made and presented to show the efficiency of our methods.

\end{itemize}

This paper is organized as follows. Some preliminaries about VIPs and the polynomial optimization problems are given in Section~\ref{sc:pre}.
In Section~\ref{sc:solveVI},
the algorithm for finding a solution to the polynomial VIP is developed. Based on the algorithm, we further introduce methods for finding more solutions to the polynomial VIP.
Section~\ref{sc:pop} studies the Moment-SOS hierarchy of semidefinite relaxations.
Numerical experiments are presented in Section~\ref{sc:ne}.
In the appendix, we show the finiteness of KKT points for solution sets of VIPs given by generic polynomials and study the algebraic degree of polynomial VIPs.

\section{Preliminaries}\label{sc:pre}

\noindent {\bf Notation}
The symbol $\mathbb{N}$ (resp., $\mathbb{R}$, $\mathbb{C}$) stands for the set of
nonnegative integers (resp., real numbers, complex numbers).
For a positive integer $k$, denote the set $[k] := \{1, \ldots, k\}$.
For a real number $t$, the $\lceil t \rceil$
represents the smallest integer not smaller than $t$.
We use $e_i$ to denote the vector such that the $i$th entry is
$1$ and all others are zeros.
By writing $A\succeq0$ (resp., $A\succ0$), we mean that the matrix $A$
is symmetric positive semidefinite (resp., positive definite).

Let $\nfR[x]$ denote the ring of polynomials
with real coefficients in $x\in {\mathbb R}^n$,
and $\nfR[x]_d$ denote its subset of polynomials whose degrees are not greater than $d$.
The $\cpx[x]$ and $\cpx[x]_d$ are defined similarly.
For a polynomial $p\in\re[x]$,
the $p=0$ means $p(x)$ is identically zero on $\re^n$.
We say the polynomial $p$ is nonzero if $p\ne0$.
Let $\af := (\af_1, \ldots, \af_n) \in \N^{n}$, and we denote
\[
x^{\af} := x_1^{\af_1} \cdots x_n^{\af_n}, \quad
\vert \af\vert :=\af_1+\cdots+\af_n.
\]
For an integer $d >0$, denote the monomial power set $\mathbb{N}_d^n \, := \,
\{\alpha\in {\mathbb{N}}^n: \, \ \vert\alpha\vert  \le d \}.$
We use $[x]_d$ to denote the vector of all monomials in $x$ in the graded alphabetical order whose degree is at most $d$.
For instance, if $x =(x_1, x_2)$, then
\[
[x]_3 = \lvt 1,  x_1, x_2, x_1^2, x_1x_2,
x_2^2, x_1^3, x_1^2x_2, x_1x_2^2, x_2^3\rvt ^T.
\]
Throughout the paper, a property is said to hold {\it generically}
if it holds for all points in the space of input data
except for a set of Lebesgue measure zero.

\subsection{Ideals and positive polynomials}
\label{ssc:poly}

Let $\mathbb{F} := \re\ \mbox{or}\ \cpx$. For a polynomial
$p \in\mathbb{F}[x]$ and subsets $I, J \subseteq \mathbb{F}[x]$,
we define
\[
p \cdot I  :=\{ p  q: \, q \in I \}, \quad
I+J  := \{a+b: \, a \in I, b \in J  \}.
\]
The subset $I$ is an ideal if $p \cdot I \subseteq I$ for all $p\in\mathbb{F}[x]$
and $I+I \subseteq I$.
For a tuple of polynomials $h = (h_1, \ldots, h_s)$, the set
\[
\idl[h]:= h_1\cdot\mathbb{F}[x] + \cdots + h_s \cdot \mathbb{F}[x]
\]
is the ideal generated by $h$, which is the smallest ideal
containing every $h_i$.

We review basic concepts in polynomial optimization.
A polynomial $\sigma \in \re[x]$ is said to be a sum of squares (SOS)
if $\sigma = p_1^2+\cdots+p_k^2$ for polynomials $p_i \in\nfR[x]~(i\in [k])$ .
The set of all SOS polynomials in $x$ is denoted as $\Sigma[x]$.
For a degree $d$, we denote the truncation
\[
\Sigma[x]_d \, := \, \Sigma[x] \cap \nfR[x]_d.
\]
For a tuple $g=(g_1,\ldots,g_t)$ of polynomials in $x$,
its quadratic module is the set
\[
\qmod[g] \, := \,  \Sigma[x] +  g_1 \cdot \Sigma[x] + \cdots + g_t \cdot  \Sigma[x].
\]
Similarly, we denote the truncation of $\qmod[g]$
\[
\qmod[g]_{2d} \, := \, \Sigma[x]_{2d} + g_1\cdot \Sigma[x]_{2d-\deg(g_1)}
+\cdots+g_t\cdot\Sigma[x]_{2d-\deg(g_t)},
\]
where $\deg(g_i)$ means the degree of $g_i$.
The tuple $g$ determines the basic closed semi-algebraic set
\begin{equation}
  \label{polyrep}
\mathcal{S}(g) \, := \,  \{x \in \nfR^n: g_1(x) \ge  0, \ldots, g_t(x) \ge  0  \}.
\end{equation}
For a tuple $h=(h_1,\ldots,h_s)$ of polynomials in $\re[x]$, its real zero set is
\[
\mc{Z}(h) := \{x \in\re^n : h_1(x)=\cdots=h_s(x)=0\}.
\]
The set $\idl[h]+\qmod[g]$ is said to be {\it archimedean}
if there exists $\rho \in \idl[h]+\qmod[g]$ such that the set $\mathcal{S}(\rho)$ is compact.
If $\idl[h]+\qmod[g]$ is archimedean, then
$\mathcal{Z}(h)\cap\mathcal{S}(g)$ must be compact.
Conversely, if $\mathcal{Z}(h)\cap\mathcal{S}(g)$ is compact, say,
$\mathcal{Z}(h)\cap\mathcal{S}(g)$ is contained in the ball $R -\|x\|^2 \ge 0$,
then $\idl[h]+\qmod[g,R -\|x\|^2]$ is archimedean
and $\mathcal{Z}(h)\cap\mathcal{S}(g) = \mathcal{Z}(h)\cap\mathcal{S}(g, R -\|x\|^2)$.
Clearly, if $f \in \idl[h]+\qmod[g]$, then
$f \ge 0$ on $\mathcal{Z}(h) \cap \mathcal{S}(g)$.
The reverse is not necessarily true.
However, when $\idl[h]+\qmod[g]$ is archimedean,
if $f > 0$ on $\mathcal{Z}(h)\cap\mathcal{S}(g)$, then $f \in \idl[h]+\qmod[g]$.
This conclusion is referenced as
Putinar's Positivstellensatz \cite{putinar1993positive}.
Interestingly, if $f \ge 0$ on $\mathcal{Z}(h)\cap\mathcal{S}(g)$,
we also have $f\in \idl[h]+\qmod[g]$,
under some standard optimality conditions \cite{nie2014optimality}.

\subsection{Localizing and moment matrices}
\label{ssc:locmat}

Let $\re^{ \N_{2d}^{n} }$ denote the space of all real vectors
that are labeled by $\af \in \N_{2d}^{n}$.
A vector $y \in \re^{ \N_{2d}^{n} }$ is labeled as
$y \,=\, (y_\af)_{ \af \in \N_{2d}^{n} }.$
Such $y$ is called a
{\it truncated multi-sequence} (tms) of degree $2d$.
For a polynomial $f  = \sum_{ \af \in \N^n_{2d} } f_\af x^\af \in  \re[x]_{2d}$,
define the operation
\be \label{<f,y>}
\langle f, y \rangle \, := \, {\sum}_{ \af \in \N^n_{2d} } f_\af y_\af.
\ee
The operation $\langle f, y \rangle$ is a bilinear function in $(f, y)$.
For a polynomial $q \in \re[x]$ with $\deg(q) \le 2d$,
and the integer $t = d - \lceil \deg(q)/2 \rceil$,
the outer product $q \cdot [x]_t ([x]_t)^T$
is a symmetric matrix polynomial in $x$, with length $\binom{n+t}{t}$.
We write the expansion as
\[
q \cdot [x]_t ([x]_t)^T \, = \, {\sum}_{ \af \in \N_{2d}^n }
x^\af  Q_\af,
\]
for some symmetric matrices $Q_\af$.
Then we define the matrix function
\be \label{df:Lf[y]}
L_{q}^{(d)}[y] \, := \, {\sum}_{ \af \in \N_{2d}^n } y_\af  Q_\af.
\ee
It is called the $d$th {\it localizing matrix} of $q$ generated by $y$.
For a given $q$, the matrix $L_{q}^{(d)}[y]$ is linear in $y$.
%Clearly, if $q(u) \geq 0$ and $y = [u]_{2d}$, then
%\[ L_{q}^{(d)}[y] = q(u) [u]_t[u]_t^T \succeq 0. \]
Localizing and moment matrices are important for getting semidefinite
relaxations of polynomial optimization problems
\cite{Las01,nie2013certifying,nie2013polynomial}.
We refer to \cite{lasserre2015introduction,LasICM,Lau09,LauICM,NieLoc}
for more references about polynomial optimization and moment problems.

\subsection{The KKT conditions and Lagrange multiplier expressions}\label{subs:KKT}
Let $x\in \SOLXF$ with the set $X$ given by (\ref{eq:X}).
Suppose $\mc{E}\cup\mc{I}=[m]$.
Under the linear independence constraint qualification condition
(LICQ, see \cite{Brks}),
there exists a Lagrange multiplier vector $\lambda=(\lambda_1 \ddd \lambda_m)$ such that
\be\label{eq:viKKT}
\left\{
\begin{array}{c}
\mbF(x) = \sum_{i=1}^{m} \lambda_i \nabla g_i(x),\
%\lambda_1\cdot g_1(x)=\dots= \lambda_m\cdot g_m(x)=0,\ g_i(x)=0\, (i\in\mc{E}),\\
g_i(x)=0\, (i\in\mc{E}),\\
\lambda_i\ge0,\ g_i(x)\ge0,\ \lambda_i\cdot g_i(x)=0 \, (i\in\mc{I}).
\end{array}
\right.
\ee
This is called the {\it KKT conditions} for the $\VIXF$.
A solution $(x,\lmd)$ to (\ref{eq:viKKT}) is called a {\it KKT pair}, and $x$ is called a {\it KKT point}.
Particularly,
when the set $X$ is convex, $x$ solves the $\VIXF$ if and only if it is a KKT point.
For the given $\mbox{VI}\left(X,F\right)$,
every KKT pair $(x,\lmd)$ satisfies
\be
\label{eq:Clmd=df}
\underbrace{\bbm
\nabla g_1(x) & \nabla g_2(x) &  \cdots &  \nabla g_m(x) \\
g_1(x) & 0  & \cdots & 0 \\
0  & g_2(x)  & \cdots & 0 \\
\vdots & \vdots & \ddots & \vdots \\
0  &  0  & \cdots & g_m(x)
\ebm}_{G(x) }
\underbrace{
\bbm  \lmd_1 \\ \lmd_2 \\ \vdots \\ \lmd_m \ebm
}_{\lmd}
=
\underbrace{\bbm  \mbF(x)  \\ 0 \\ \vdots \\ 0 \ebm}_{ \hat{\mbF}(x)} .
\ee
We say the polynomial tuple $g=(g_1,\ldots,g_m)$ is nonsingular if the matrix of polynomials $G(x)$ has full column rank for all $x\in\cpx^n$.
By \cite[Proposition~5.1]{Nie2019},
the tuple $g$ is nonsingular if and only if
there exists a matrix of real coefficient polynomials $L(x)$ such that
\be\label{eq:LG=I}L(x)G(x)=I_m.\ee
When constraining polynomials $g_1\ddd g_m$ are generic,
$\rank\, G(x)=m$  for all $x\in\cpx^n$ \cite[Proposition~5.7]{Nie2019},
and there exists an $L(x)$ satisfying (\ref{eq:LG=I}).
In this case, the following identity holds at all KKT pairs $(x,\lmd)$:
\be\label{eq:LME}\lambda=L(x)\cdot\hat{\mbF}(x).\ee
Denote the $i$th component of $L(x)\cdot\hat{\mbF}(x)$ by $\lmd_i(x)$,
and let $\lmd(x):=(\lmd_1(x) \ddd \lmd_m(x))$.
Then $\lmd(x)$ is a tuple of real polynomials in $x$ and $\lambda_i=\lmd_i(x)$ for all $i\in[m]$ at all KKT pairs $(x,\lambda)$.
The $\lmd(x)$ is called a {\it polynomial Lagrange multiplier expression} (LME) for the $\VIXF$.

For the polynomial $\VIXF$, suppose the LME $\lmd(x)$ exists.
Then the KKT system (\ref{eq:viKKT}) can be equivalently written as
\be\label{eq:viKKTLME}
\left\{
\begin{array}{c}
\mbF(x) = \sum_{i=1}^{m} \lambda_i(x) \nabla g_i(x),\\
g_i(x)=0\, (i\in\mc{E}),\\
\lambda_i(x)\ge0, g_i(x)\ge0, \lambda_i(x)\cdot g_i(x)=0 \, (i\in\mc{I}).
\end{array}
\right.
\ee
This is a system of polynomial equalities and inequalities in $x$.
If we denote by $\mc{K}$ the set of points satisfying (\ref{eq:viKKTLME}),
then $\SOLXF\subseteq\mc{K}$.
For the convenience of our discussion,
we denote
\[
 E \,\coloneqq\, \{\mbF(x) - \sum\nolimits_{i=1}^{m} \lmd_i(x) \nabla g_i(x)\}\cup \{\lmd_i(x)\cdot g_i(x) \, \vert \, i\in \mc{I} \} \cup \{g_i(x) \, \vert \, i\in \mc{E} \},
\]
\[
 I \,\coloneqq\, \{\lmd_i(x) \, \vert \, i\in \mc{I} \} \cup \{g_i(x) \, \vert \, i\in \mc{I} \}.
\]
Then, $\mc{K}=\mc{Z}( E )\cap\mc{S}( I )$.

\section{Finding solutions for polynomial VIPs}
\label{sc:solveVI}
In this section, we first propose an algorithm for finding a solution to the polynomial $\VIXF$ by solving a sequence of polynomial optimization problems.
Then, based on the proposed algorithm, we provide a method to find more solutions to the $\VIXF$.
%and to solve optimization problems with variational inequality constraints.
Throughout this section, we assume the $\VIXF$ is given by polynomials and the constraining tuple $g=(g_1\ddd g_m)$ is nonsingular, thus LMEs exist.

\subsection{Finding one solution for the polynomial VIP}
\label{sc:find1}
For the given polynomial $\VIXF$,
let $\Theta$ be a generic positive definite matrix. Consider the following polynomial optimization problem ($[x]_1:=(1,x_1,x_2\ddd x_n)$)
\be\label{eq:KKTopt}
\left\{\ \begin{array}{ll} \min\limits_{x\in\re^n} \quad & [x]_1^T\Theta[x]_1,\\
\st& x\in\mc{K}.
\end{array}\right.
\ee
% where the vector $[x]_1=(1,x_1,x_2\ddd x_n)$.
For the $\VIXF$ with a convex feasible set $X$,
by solving (\ref{eq:KKTopt}),
one gets a solution to the $\VIXF$ if (\ref{eq:KKTopt}) has a nonempty feasible set,
or detect the emptiness of $\SOLXF$ otherwise.
This is shown in the following proposition.
\begin{prop}
\label{prop:cvx}
Consider the polynomial $\VIXF$, where $X$ is a convex set that is given by (\ref{eq:X}) with a nonsingular polynomial tuple $g$.
If the problem (\ref{eq:KKTopt}) is infeasible,
then $\VIXF$ does not have any solution.
Otherwise, the solution set of (\ref{eq:KKTopt}) is nonempty, and every minimizer solves the $\VIXF$.
\end{prop}
\begin{proof}
By \cite[Proposition~5.1]{nie2014moment}, the nonsingularity of the polynomial tuple $g$ yields the existence of the LMEs, which also infers that the LICQ holds. Thus $\SOLXF\subseteq\mc{K}$, and the infeasibility of the problem (\ref{eq:KKTopt}) implies the emptiness of $\SOLXF$.

Suppose the problem (\ref{eq:KKTopt}) is feasible, i.e., $\mc{K}\ne\emptyset$.
Since $\Theta$ is a generic positive definite matrix,
the quadratic polynomial $[x]_1^T\Theta[x]_1$ is strongly convex,
which implies the existence of minimizers for (\ref{eq:KKTopt}).
In addition, under the nonsingularity condition of $g$ and the convexity of $X$,
we have $\SOLXF=\mc{K}$, thus the minimizer of (\ref{eq:KKTopt}) solves $\VIXF$.
\qed\end{proof}
However, when the set $X$ is nonconvex, a minimizer of (\ref{eq:KKTopt}) may not solve the $\VIXF$. Consider the gap function
\[\nu(x)\defeq \inf_{y\in X}\ (y-x)^T\mbF(x).\]
It is clear that $\nu(u)\leq 0$ for all $u\in X$, and $u\in\SOLXF$ if and only if
$\nu(u)= 0$.
If $u\notin\SOLXF$, the gap function can be used to construct an extra constraint to shrink the feasible set of (\ref{eq:KKTopt}).
That is, for a minimizer $u$ of (\ref{eq:KKTopt}) such that $u\notin\SOLXF$, we consider
\be\label{eq:linopt}
\left\{\
\begin{array}{lll}
\displaystyle\varepsilon\,:=\ & \min\limits_{y\in \re^n}\quad  & (y-u)^T\mbF(u)\\
& \st & y\in X.
\end{array}\right.
\ee
To our main interest,
assume (\ref{eq:linopt}) has a minimizer $v$ (otherwise, we may add some extra boundedness constraints, see Section~\ref{sc:popverify}).
Let $\mc{K}':=\mc{K}\cap\{x\in\re^n:(v-x)^T\mbF(x) \geq 0\},$
then we have $\SOLXF\subseteq \mc{K}'$ and $u\notin \mc{K}'$. Therefore, if we solve (\ref{eq:KKTopt}) with a replacement of $\mc{K}$ by the shrinked feasible set $\mc{K}'$,
then another candidate solution to the $\VIXF$ will be obtained.
Indeed, we can repeat this procedure to remove KKT points that are not solutions,
and get tighter approximations to the solution set $\SOLXF$,
until an actual solution is obtained or the nonexistence of solutions is detected.
Based on such a scheme, we propose the following algorithm for finding a solution to the $\VIXF$:

\begin{alg}
\label{ag:KKTSDP} \rm
For $\VIXF$ such that $X$ is given by (\ref{eq:X}) with nonsingular $g$,
do the following:

\begin{description}

\item [Step~0]
Choose a generic positive definite matrix $\Theta$ of length $n+1$.
Set $k:=1$ and $V:=\emptyset$.

\item [Step~1]
Solve the following polynomial optimization problem
\be\label{eq:genopt}
\left\{\
\begin{array}{ll}
\displaystyle\min\limits_{x\in\re^n} \quad  & [x]_1^T\Theta[x]_1\\
\st & x\in \mc{K},\\
    & (v-x)^T\mbF(x) \geq 0~(v\in V).
\end{array}\right.
\ee
If (\ref{eq:genopt}) is infeasible, then $\SOLXF$ is empty and stop;
otherwise, solve (\ref{eq:genopt}) for a minimizer $u^{(k)}$.

\item [Step~2]
Solve the polynomial optimization problem \eqref{eq:linopt} with $u:=u^{(k)}$ for a set of minimizers $T^{(k)}$.

\item [Step~3]
If $\varepsilon= 0$, then $u^{(k)}\in\SOLXF$ and stop; otherwise, let $V:=V\cup T^{(k)}$ and $k:=k+1$, then go back to Step~1.

\end{description}
\end{alg}
\begin{rmk}
(i) In Section~\ref{sc:pop}, we will study how to solve the occurring polynomial optimization subproblems in Algorithm~\ref{ag:KKTSDP}.
(ii) In Step 1, all feasible points of (\ref{eq:genopt}) are KKT points.
Thus, the objective function in (\ref{eq:genopt}) can be replaced by any polynomial function in $x$.
However, using $[x]_1^T\Theta[x]_1$ gives us better convergence results for Algorithm~\ref{ag:KKTSDP},
and makes (\ref{eq:genopt}) easier to solve; see Theorem~\ref{tm:findone} and Section~\ref{sc:findcan}.
(iii) For (\ref{eq:genopt}), we may also replace the constraint $x\in\mc{K}$ by the KKT system (\ref{eq:viKKT}), which means we treat all $\lmd_i$ as new variables without using LMEs.
However, this is not computationally efficient, as shown in Examples~\ref{ep:Kojima} and \ref{ep:Kojima_nosol} in Section \ref{sc:epVI}.
(iv) We do not need the set $T^{(k)}$ in Step~2 to contain all minimizers of \eqref{eq:linopt}.
Indeed, one minimizer is enough to get another solution candidate.
However, we can usually get a tighter approximation to $\SOLXF$ if there are more points in $T^{(k)}$.
When \eqref{eq:linopt} has more than one minimizer, the {\it flat truncation} can be applied to find them; see Section~\ref{sc:pop} for more details.
\end{rmk}

We say Algorithm~\ref{ag:KKTSDP} terminates within finitely many steps if the algorithm finds a solution to $\VIXF$ or detects the nonexistence of solutions at some iteration loop $k$.
In the following, we establish the convergence property for Algorithm~\ref{ag:KKTSDP}.

\begin{thm}
\label{tm:findone}
Consider the $\VIXF$ where $X$ is given by (\ref{eq:X}) with a nonsingular polynomial tuple $g$.
Let $\mathfrak{s}:=\vert \mc{K}\setminus\SOLXF\vert.$
\begin{enumerate}[(i)]
  \item If $\SOLXF=\emptyset$ and $\mathfrak{s}<\infty$, then Algorithm~\ref{ag:KKTSDP} will detect nonexistence of solutions to the $\VIXF$ with $k\le\mathfrak{s}+1$.
  \item If $\SOLXF\ne \emptyset$ and $\mathfrak{s}<\infty$, then Algorithm~\ref{ag:KKTSDP} will find a solution to the $\VIXF$ with $k\le\mathfrak{s}+1$.
  \item Suppose $\mathfrak{s}=\infty$ and Algorithm~\ref{ag:KKTSDP} does not terminate within finitely many steps. If $u^*$ is an accumulation point of $\{u^{(k)}\}_{k=1}^{\infty}$, the gap function $\nu(u)$ is continuous at $u^*$, and $\lim\sup_{k\to\infty}T^{(k)}$ is bounded, then $u^*\in\SOLXF$.
\end{enumerate}
\end{thm}
\begin{proof}
We assume $\mc{K}$ is nonempty. Otherwise, the conclusion follows directly from the fact that $\vert \mc{K}\vert =0$ implies the infeasibility for (\ref{eq:KKTopt}).
Denote by $\mc{K}_k$ the feasible set of (\ref{eq:genopt}) at the $k$th iteration. Then, $\mc{K}_{k+1}\subseteq\mc{K}_{k}$ for all $k\in\N_+$. In addition, if Algorithm~\ref{ag:KKTSDP} does not terminate at the $k$th loop, then $u^{(k)}\notin\mc{K}_{k+1}$. Thus,
we have the following strictly descent chain:
\[\mc{K}=\mc{K}_1\supsetneq \mc{K}_2 \dots \supsetneq \mc{K}_k\supsetneq \dots\supsetneq \SOLXF.\]
Since $\mathfrak{s}=\vert \mc{K}\setminus\SOLXF\vert <\infty$, there exists $k_1\in\N_+$ such that $k_1\le\mathfrak{s}+1$ and $\mc{K}_{k_1}=\SOLXF$.

(i) If $\SOLXF=\emptyset$, then problem (\ref{eq:genopt}) with $k=k_1$ is infeasible, which means the $\VIXF$ does not have any solution.
Therefore, Algorithm~\ref{ag:KKTSDP} can detect the nonexistence of solutions to $\VIXF$ in at most $k_1$ loops.

(ii) If $\SOLXF\ne\emptyset$, then every feasible point of (\ref{eq:genopt}) with $k=k_1$ is a solution to $\VIXF$.
Therefore, Algorithm~\ref{ag:KKTSDP} can find a solution to $\VIXF$ in at most $k_1$ loops.

(iii)
Assume that Algorithm~\ref{ag:KKTSDP} does not terminate within finitely many steps,
and there exists an accumulation point $u^*$ of sequence $\{u^{(k)}\}_{k=1}^{\infty}$.
Without loss of generality, we assume $u^{(k)}\to u^*$ as $k\to\infty$.
Note that $u^*$ is a solution to $\VIXF$ if and only if $\nu(u^*)=0$.
For each $i=1,2,\dots$,
let $v^{(i)}$ be a minimizer of \eqref{eq:linopt} for $u:=u^{(i)}$.
Since $u^{(i)}\notin\SOLXF$, we have \[(v^{(i)}-u^{(i)})^T\mbF(u^{(i)})=\nu(u^{(i)})<0.\]
For each $j>i$, based on the feasibility of \eqref{eq:genopt} at $u^{(j)}$,
we have
\[(v^{(i)}-u^{(j)})^T\mbF(u^{(j)})\ge0.\]
Due to the continuity of $\mbF$, we have the $(v^{(i)}-x)^T\mbF(x)$ is continuous in $x$
for every given $v^{(i)}$. Hence,
\[(v^{(i)}-u^*)^T\mbF(u^*)\ge0\fall i\in \N_+.\]
Moreover, we have the following for all $i\in \N_+.$
\[\begin{split}
\nu(u^*)&=\nu(u^*)-\nu(u^{(i)})+\nu(u^{(i)})\\
        &\ge \nu(u^*)-\nu(u^{(i)})+(v^{(i)}-u^{(i)})^T\mbF(u^{(i)})-(v^{(i)}-u^*)^T\mbF(u^*).
\end{split}\]
From the assumption that $\lim\sup_{k\to\infty}T^{(k)}$ is bounded,
there exists a subsequence of $\{v^{(j)}\}_{j=1}^{\infty}$, say $\{v^{(k_j)}\}_{j=1}^{\infty}$, being convergent.
Denote by $v^*$ the limit point of $\{v^{(k_j)}\}_{j=1}^{\infty}$.
Then, we have
\be\label{eq:nuge}
\nu(u^*)\ge \nu(u^*)-\nu(u^{(k_j)})+(v^{(k_j)}-u^{(k_j)})^T\mbF(u^{(k_j)})-(v^{(k_j)}-u^*)^T\mbF(u^*)
\ee
hold for all $j\in \N_+$.
Let $j\to\infty$, the right hand side of (\ref{eq:nuge}) equals to $0$,
because $\nu$ is continuous at $u^*$,
and $(v-u)^T\mbF(u)$ is continuous in variable of $(u,v)$.
So, we have that $\nu(u^*)\ge0$. Besides, $\nu(u)\le 0$ for all $u\in X$ based on the definition. Therefore, we conclude that $\nu(u^*)=0$, which means $u^*\in\SOLXF$.
\qed\end{proof}
We remark that if the sequence $\{u^{(k)}\}_{k=1}^{\infty}$ generated by Algorithm~\ref{ag:KKTSDP} has no accumulation point, then $\SOLXF=\emptyset$.
To see this, if $\{u^{(k)}\}_{k=1}^{\infty}$ does not have any accumulation point,
then it must be unbounded.
Note that the objective function $\theta(x):=[x]_1^T\Theta[x]_1$ is a quadratic positive definite polynomial.
The $\theta(u^{(k)})$ goes to $\infty$ as $k\to \infty$.
Suppose $x^*\in\SOLXF$.
Then, we have $\infty>\theta(x^*)\ge \theta(u^{(1)})$,
otherwise $u^{(1)}$ cannot be the minimizer for (\ref{eq:KKTopt}).
Since $\lim_{k\to\infty}\theta(u^{(k)})=\infty$,
there exists $k_1\in\N$ such that
$\theta(u^{(k_1)})\le \theta(x^*) <\theta(u^{(k_1+1)}).$
Because $x^*\in\SOLXF$,
$x^*$ satisfies the KKT condition and for any $v\in X$ the following holds
\[(v-x^*)^T\mbF(x^*)\ge 0.\]
It means that $x^*$ is feasible for (\ref{eq:genopt}) at the $(k_1+1)$th loop,
which contradicts to $u^{(k_1+1)}$ being a minimizer of (\ref{eq:genopt}).
So such $x^*$ does not exist and we conclude the emptiness of $\SOLXF$ if the sequence $\{u^{(k)}\}_{k=1}^{\infty}$ has no accumulation point.

The continuity condition of $\nu(x)$ holds under certain conditions, for example, the {\it restricted inf-compactness} (RIC) condition;
see \cite[Definition~3.13]{GuoLinYeZhang}.
However, to implement Algorithm~\ref{ag:KKTSDP} does not require any {\it a priori} knowledge on the finiteness of $\mathfrak{s}$ or the continuity of $\nu$.
Indeed, even if $\mathfrak{s}=\infty$, Algorithm~\ref{ag:KKTSDP} may still terminate in finitely many steps.
In all numerical experiments in Section~\ref{sc:ne}, our method finds solutions to the VIPs or detects the nonexistence of solutions in finitely many steps.
To end this section, we demonstrate the following result saying that if the $\VIXF$ is given by generic polynomials $\mbF$ and $g$,
then $\mc{K}$ is a finite set,
which guarantees the finite termination for Algorithm~\ref{ag:KKTSDP}.
The proof is given in the appendix.

\begin{thm}
\label{tm:solfinite}
Let $a_1,a_2\ddd a_n,b_1,b_2\ddd b_m\in\N_+$ be degrees.
Suppose the $\VIXF$ is defined by polynomials such that $F_i(x)$ is a generic polynomial in $\cpx[x]_{a_i}$ for all $i\in [n]$,
and $g_j(x)$ is a generic polynomial in $\cpx[x]_{b_j}$ for all $j\in [m]$.
Then, every $x\in\SOLXF$ is a KKT point of the $\VIXF$,
and we have $\vert \mc{K}\vert <\infty$.
\end{thm}

\subsection{Finding more solutions to the polynomial VIP}
\label{sc:findmore}

In this subsection, we further investigate how to find more solutions or check the completeness of the computed solution set for the $\VIXF$, based on Algorithm~\ref{ag:KKTSDP}.

Suppose Algorithm~\ref{ag:KKTSDP} terminates at a solution $x^*$,
and $x^*$ is an isolated point in the feasible region at the last iteration of Algorithm~\ref{ag:KKTSDP}.
Since the matrix $\Theta$ in the objective function of \reff{eq:genopt} is generic and positive definite, by \cite[Proposition~5.2]{nie2014moment},
there exists $\delta>0$ such that for all $x\in\SOLXF\setminus\{x^*\}$,
\be\label{eq:moreoptbd}
[x]_1^T\Theta[x]_1 \ge [x^*]_1^T\Theta [x^*]_1+\delta.
\ee
Therefore, if we add \reff{eq:moreoptbd} as an extra constraint to \reff{eq:genopt},
then we will get a new candidate solution for the $\VIXF$ other than $x^*$.
That is, we solve the following polynomial optimization problem:
\be\label{eq:moreopt}
\left\{\
\begin{array}{ll}
\displaystyle\min\limits_{x\in\re^n} \quad  & [x]_1^T\Theta[x]_1\\
\st & x\in \mc{K}_{x^*},
\end{array}\right.
\ee
where \[\mc{K}_{x^*}:=\{x\in \mc{K}~\vert ~(v-x)^T\mbF(x) \geq 0~(v\in V),~[x]_1^T\Theta[x]_1 \ge [x^*]_1^T\Theta [x^*]_1+\delta\}.\]
If (\ref{eq:moreopt}) is infeasible, then $\SOLXF=\{x^*\}$.
Otherwise, (\ref{eq:moreopt}) must have a minimizer $u$ due to the positive definiteness of $\Theta$,
and we can check whether $u\in\SOLXF$ or not by solving (\ref{eq:linopt}).

In practical implementation,
one needs a suitable $\delta$ to formulate (\ref{eq:moreopt}).
Given $x^*$ and $\mc{K}_{x^*}$, the following trick is useful for finding the $\delta$.
By \cite[Proposition~3.5]{Nie2020nash} (see also \cite{NieLoc}), if $x^*$ is isolated in $\mc{K}_{x^*}$,
then there exists $\hat{\delta}>0$ such that
\[\mc{K}_{x^*}\cap\{x\in\re^n:[x]_1^T\Theta[x]_1\le [x^*]_1^T\Theta [x^*]_1+\hat{\delta}\}=\{x^*\}.\]
Moreover, since $\SOLXF\subseteq\mc{K}_{x^*}$,
all $\dt$ satisfying $0<\dt\le\hat{\dt}$ is available for (\ref{eq:moreopt}).
For {\it a priori} $\dt_0>0$, let $\dt=\dt_0$ and
consider the following maximization problem
\be\label{eq:max}
\left\{\ \begin{array}{lll}
\gamma^*:=&\max\limits_{x\in\re^n}\quad & [x]_1^T\Theta[x]_1\\
&\st & x\in \mc{K},~ (v-x)^T\mbF(x)\geq 0\, (v\in V),\\
&& [x]_1^T\Theta[x]_1 \le [x^*]_1^T\Theta [x^*]_1+\delta.
\end{array}\right.
\ee
If $\gamma^*=[x^*]_1^T\Theta [x^*]_1$,
then $\dt$ is applicable for (\ref{eq:moreopt}).
Otherwise, we update $\dt:=\rho\cdot\dt$ for some $\rho<1$ and solve (\ref{eq:max}) again.
Repeating these steps, we may find more solutions to the $\VIXF$.

\begin{alg}
% \caption{An algorithm for finding another solution to $\VIXF$.}
\label{ag:more} \rm
Under the same assumptions as in Algorithm~\ref{ag:KKTSDP},
suppose we have found a solution $x^*$ to the $\VIXF$.
Do the following:

\begin{description}

\item [Step~0]
Choose a $\dt_0>0$ and a positive $\rho<1$.
Let $\dt\defeq \dt_0$.

\item [Step~1]
Solve the polynomial optimization \reff{eq:max} for the maximal value $\gm^*$.

\item [Step~2]
If $\gamma^*=[x^*]_1^T\Theta [x^*]_1$, then proceed to the next step;
otherwise, let $\dt:=\dt\cdot\rho$ and go back to Step~1.

\item [Step~3]
Solve the polynomial optimization \reff{eq:moreopt}.
If it is infeasible, then there does not exist any other solution for $\VIXF$.
Otherwise, solve \reff{eq:moreopt} for a minimizer $u$.

\item [Step~4]
Solve the linear optimization problem (\ref{eq:linopt}) with $u$ for a set of minimizers $T$.

\item [Step~5]
If $\varepsilon=0$, then $u\in\SOLXF$ is another solution and stop;
otherwise, let $V:=V\cup T$, then go back to Step~3.

\end{description}
\end{alg}
When $x^*\in\SOLXF$ is isolated in $\mc{K}_{x^*}$,
we can always find a suitable $\dt$ for (\ref{eq:genopt}) by solving (\ref{eq:max}) for finitely many times.
Therefore, if $|\mc{K}\setminus \SOLXF|$ is finite,
Algorithm~\ref{ag:more} will surely produce a new solution within finitely many steps, by Theorem~\ref{tm:findone}.
Indeed,
when $\Theta$ is generic and all points in $\SOLXF$ are isolated,
we have the following hierarchy of solutions
\begin{multline}
\label{eq:hiersol}
[x^{(1)}]_1^T\Theta[x^{(1)}]_1<
[x^{(2)}]_1^T\Theta[x^{(2)}]_1<
[x^{(3)}]_1^T\Theta[x^{(3)}]_1<\cdots\\
\mbox{with} \quad \SOLXF=\{x^{(1)},x^{(2)},x^{(3)},\dots\}.
\end{multline}
As listed in (\ref{eq:hiersol}), the solution first obtained by Algorithm~\ref{ag:KKTSDP} is $x^{(1)}$,
and Algorithm~\ref{ag:more} finds $x^{(k)}$ with $x^*:=x^{(k-1)}$.
Particularly,
when $\vert \SOLXF\vert =:r<\infty$, such a scheme finds all solutions to $\VIXF$,
as long as $x^{(k)}$ is isolated in $\mc{K}_{x^{(k-1)}}$ for each $k<r$.
Furthermore, if we apply Algorithm~\ref{ag:more} with $x^*:=x^{(r)}$,
then the algorithm terminates at {Step~3} because the feasible set $\mc{K}_{x^*}$ of (\ref{eq:moreopt}) must be empty for some $V$,
and we conclude that a complete solution set is obtained.
\begin{thm}\label{tm:all}
Consider the $\VIXF$ where $X$ is given by (\ref{eq:X}) with a nonsingular polynomial tuple $g$.
If $\vert \mc{K}\vert <\infty$, then all solutions to $\VIXF$ can be obtained by running Algorithm~\ref{ag:more} finitely many times.
\end{thm}
It is clear that when the $\VIXF$ is given by generic polynomials,
we can find all solutions to the $\VIXF$ because $\vert \mc{K} \vert <\infty$, by Theorems~\ref{tm:solfinite} and \ref{tm:all}.

\section{Solving the polynomial optimization subproblems}
\label{sc:pop}
In this section, we study the method for handling polynomial optimization subproblems appearing in Algorithms~\ref{ag:KKTSDP} and~\ref{ag:more}.
In general, we consider a polynomial optimization problem in the following form
\be\label{eq:pop}
\left\{\ \begin{array}{lll}
\vartheta\defeq &\min\limits_{x\in\re^n}\quad & \theta(x)\\
&\st & p(x)=0\, (p\in\Phi),\\
&    & q(x)\ge0\, (q\in\Psi),
\end{array}
\right.
\ee
where $\theta(x)$ is a polynomial function in $x$,
and $\Phi,\Psi$ are tuples of polynomials in $x$.

We introduce the Moment-SOS hierarchy of semidefinite relaxations for (\ref{eq:pop}).
Denote
\be\label{eq:d0}
d_0 \, := \, \max\{\lceil\deg(p)/2\rceil: \, p \in \Phi \cup \Psi\cup \{\theta\} \}.
\ee
For a degree $k \ge d_0$, the $k$th order moment relaxation for solving (\ref{eq:pop}) is
\be
\label{eq:d-mom}
\left\{
\baray{rl}
\vartheta_k  \, :=  \,\min\limits_{ y } & \lip \theta,y \rip\\
 \st & y_0=1,\, L_{p}^{(k)}[y] = 0 \, (p \in \Phi), \\
  & M_k[y] \succeq 0,\, L_{q}^{(k)}[y] \succeq 0 \, (q \in \Psi),  \\
  &  y \in \mathbb{R}^{\mathbb{N}^{n}_{2k}},
\earay
\right.
\ee
where $M_k[y]$ is the $k$th order moment matrix generated by $y$, and $L_{p}^{(k)}[y]$, $L_{q}^{(k)}[y]$ are $k$th localizing matrices of $p,q$ generated by $y$. Its dual problem is the $k$th order SOS relaxation for (\ref{eq:pop})
\be
\label{eq:d-sos}
\left\{
\baray{lll}
\gamma_k\defeq & \max & \gamma\\
&\st & \theta -\gamma \in  \idl[\Phi]_{2k}+\qmod[\Psi]_{2k}. \\
\earay
\right.
\ee
For relaxation orders $k=d_0, d_0+1, \ldots$,
we get the Moment-SOS hierarchy of semidefinite relaxations
\reff{eq:d-mom}-\reff{eq:d-sos}.
The optimization \reff{eq:d-mom} is a relaxation of \reff{eq:pop}.
In fact, if $x$ is a feasible point of \reff{eq:pop},
then $y = [x]_{2k}$ is feasible for \reff{eq:d-mom}, while the infeasibility of \reff{eq:d-mom} implies the infeasibility of \reff{eq:pop}.
On the other hand,
$\vartheta_k$ and $\gamma_k$ provide lower bounds for $\vartheta$ for every $k$,
that \[\gamma_k\le \vartheta_k \le \vartheta.\]
Let $e_i$ be the labeling vector such that its $i$th entry is $1$ while all other entries are $0$s,
and $y_{e_i}$ be the subvector of the tms $y$ whose entries are labeled by $e_i$
(for instance, when $n=4$,  $y_{e_3} = y_{0010}$).
For a given $k$, let $y^{(k)}$ be a minimizer of the $k$th moment relaxation (\ref{eq:d-mom}), and
$u^{(k)}\,\coloneqq\, (y^{(k)}_{e_1} \ddd y^{(k)}_{e_n}).$
Then $u^{(k)}$ must be a minimizer of \reff{eq:genopt} if (we refer to Section~\ref{sc:pre} for definitions of $\mc{Z}$ and $\mc{S}$)
\be \label{eq:uniquemin}
u^{(k)}\in\mc{Z}(\Phi)\cap\mc{S}(\Psi),\quad \vartheta_k = \theta(u^{(k)}).
\ee
However, (\ref{eq:uniquemin}) usually does not hold when there is more than one minimizer.
For such cases, the {\it flat truncation} can be applied to certify global optimality and obtain minimizers.
To conclude, we summarize the following algorithm, namely, the Moment-SOS hierarchy of semidefinite relaxations, for solving polynomial optimization problems:
\begin{alg} \label{ag:momwithFT} \rm
% \caption{The Moment-SOS hierarchy of semidefinite relaxations}
For the polynomial optimization problem \reff{eq:pop}.
Initialize $k := d_0$.
\begin{description}

\item [Step~1]
Solve the moment relaxation (\ref{eq:d-mom})
for the minimum value $\vartheta_k$.
If (\ref{eq:d-mom}) is infeasible,
then the polynomial optimization problem is infeasible.
Otherwise, solve (\ref{eq:d-mom}) for a minimizer $y^*$ and let $t:=d_0$.

\item [Step~2]
Let $u^*\,\coloneqq\, (y^*_{e_1} \ddd y^*_{e_n})$.
If $u^*$ satisfies (\ref{eq:uniquemin}),
then $u^*$ is the minimizer for (\ref{eq:pop}) and stop.

\item [Step~3]
If $y^*$ satisfies the rank condition
\be \label{eq:flatrank}
 \Rank{M_t[y^*]} \,=\, \Rank{M_{t-d_0}[y^*]} ,
\ee
then extract a set $U_i$ of
$r :=\Rank{M_t[y^*]}$ minimizers for (\ref{eq:pop}) and stop.

\item [Step~4]
If \reff{eq:flatrank} fails to hold and $t < k$,
let $t := t+1$ and then go to Step~3;
otherwise, let $k := k+1$ and go to Step~1.

\end{description}
\end{alg}

For Algorithm~\ref{ag:momwithFT},
we say it has asymptotic convergence if $\vartheta_k\to \vartheta$ as $k\to\infty$,
and we say it has finite convergence if there exists $K\in\N$ such that $\vartheta_k= \vartheta$ for all $k\ge K$.
In Step~3, the rank condition~\reff{eq:flatrank} is called \textit{flat truncation} \cite{nie2013certifying}.
It is a sufficient (and almost necessary) condition to check the finite convergence
of moment relaxations.
When \reff{eq:flatrank} holds, the method in \cite{HenLas05}
can be used to extract $r$ minimizers for \reff{eq:pop}.
This method and Algorithms~\ref{ag:momwithFT} are implemented in the software {\tt GloptiPoly 3}~\cite{GloPol3}.
In the following subsections, we discuss how to use Algorithm~\ref{ag:momwithFT} for solving polynomial optimization problems (\ref{eq:linopt}-\ref{eq:genopt}) and (\ref{eq:moreopt}-\ref{eq:max}), respectively.

\subsection{Finding solution candidates to polynomial VIPs}
\label{sc:findcan}
In this subsection, we discuss how to solve the polynomial optimization for obtaining candidate solutions to polynomial VIPs, particularly for solving the problem (\ref{eq:genopt}).
Note that (\ref{eq:genopt}) can be rewritten as (\ref{eq:pop}) with
\be\label{eq:pp2}
\left\{
\begin{aligned}
\ \Phi&\defeq  E,\quad \theta(x)\defeq [x]_1^T\Theta[x]_1,
\\
%%%%
\ \Psi&\defeq  I  \cup\{ (v-x)^T\mbF(x)\, :\, v\in V\}.
\end{aligned}\right.
\ee
Since the $\theta(x)$ is a generic positive definite quadratic polynomial,
the problem (\ref{eq:genopt}) has a unique minimizer, as we show in Theorem~\ref{thmcvg:upconv}.
In practice, we usually let $\theta(x)\defeq [x]_1^TR^TR[x]_1$ with $R$ being a randomly generated $(n+1)\times (n+1)$ matrix.
As we mentioned before,
the infeasibility of (\ref{eq:d-mom}) infers the infeasibility of (\ref{eq:genopt}),
which further implies the nonexistence of the KKT points.
Moreover, when \reff{eq:genopt} is feasible, the optimal value $\vartheta_k$ of \reff{eq:d-mom}
is a lower bound for the minimum value of \reff{eq:genopt},
i.e., $\vartheta_k \leq [x]_1^T\Theta[x]_1$ for all $x$
that is feasible for \reff{eq:genopt}.

% The Algorithm~\ref{ag:KKTopt} can be implemented in {\tt GloptiPoly} \cite{GloPol3}.
The convergence of Algorithm~\ref{ag:momwithFT} for solving (\ref{eq:genopt}) is shown as follows.

\begin{thm} \label{thmcvg:upconv}
Assume the set $\idl[\Phi]+\qmod[\Psi] \subseteq \re[x]$ is archimedean.
\bit

\item [(i)]
If \reff{eq:genopt} is infeasible,
then the moment relaxation \reff{eq:d-mom}
is infeasible when $k$ is big enough.

\item [(ii)]
Suppose \reff{eq:genopt} is feasible and
$\Theta$ is a generic positive definite polynomial.
Then \reff{eq:genopt} has a unique minimizer $u$,
and $u^{(k)}\to u$ as $k\to\infty$.
Furthermore,
if $\Phi(x)=0$ have finitely many real solutions,
then Algorithm~\ref{ag:momwithFT} has finite convergence with (\ref{eq:flatrank}) satisfied.
\eit
\end{thm}
\begin{proof}
(i) If \reff{eq:genopt} is infeasible,  the constant polynomial $-1$
can be viewed as a positive polynomial on the feasible set of \reff{eq:genopt}.
Since $\idl[\Phi]+\qmod[\Psi]$ is archimedean,
we have $-1 \in  \idl[\Phi]_{2k}+\qmod[\Psi]_{2k}$,
for $k$ big enough, by the Putinar Positivstellensatz \cite{putinar1993positive}.
For such a big $k$, the SOS relaxation \reff{eq:d-sos} is unbounded from above,
hence the moment relaxation \reff{eq:d-mom} must be infeasible.

(ii) When the problem \reff{eq:genopt} is feasible and $\Theta$ is generic positive definite,
it must have a unique minimizer, by \cite[Proposition~5.2]{nie2014moment}.
Let $u$ be this minimizer.
The convergence of $ u^{(k)}$ to $u$ is shown in
\cite{Swg05} or \cite[Theorem~3.3]{nie2013certifying}.
For the special case that $\Phi(x)=0$ has finitely many real solutions,
the point $ u^{(k)}$ must be equal to $u$ and (\ref{eq:flatrank}) holds,
when $k$ is large enough.
This is shown in \cite{lasserre2008semidef} (see also \cite{nie2013polynomial}).
\qed\end{proof}
\begin{rmk}
(i) By Theorem~\ref{thmcvg:upconv},
when $\idl[\Phi]+\qmod[\Psi] \subseteq \re[x]$ is archimedean,
$u^{(k)}$ converges to the minimizer of (\ref{eq:genopt}).
If the point $ u^{(k)}$ is feasible and $\vartheta_k=\theta( u^{(k)})$,
then $ u^{(k)}$ must be a minimizer of (\ref{eq:genopt}),
regardless of the archimedeaness holds or not.
Moreover, the infeasibility of (\ref{eq:d-mom}) always implies the infeasibility of (\ref{eq:genopt}),
regardless of the archimedeaness holds or not.

(ii) We would like to remark that if flat truncation (\ref{eq:flatrank}) is satisfied,
then \reff{eq:uniquemin} holds and $\rank\,M_t[y^*]=1$.
However, it is possible that the $u^{(k)}$ is a minimizer for (\ref{eq:genopt}), while (\ref{eq:flatrank}) does not hold (see \cite{nie2013polynomial}).
When $\Phi(x)=0$ has finitely many real solutions (even when archimedeaness does not hold),
flat truncation holds for all $k$ that is big enough with $\rank\,M_t[y^*]=1$ \cite{lasserre2008semidef}.
In the actual implementation of solving (\ref{eq:genopt}), we usually skip Step 3 in Algorithm~\ref{ag:momwithFT}.

(iii) By Theorem~\ref{tm:solfinite}, the $\Phi(x)=0$ defines a finite set in $\re^n$ when the $\VIXF$ is given by general polynomials.
For such cases, finite convergence of Algorithm~\ref{ag:momwithFT} is guaranteed.
Moreover, in our computational practice,
Algorithm~\ref{ag:momwithFT} always has finite convergence.
\end{rmk}

The polynomial optimization subproblem (\ref{eq:moreopt}) can be solved in the same
way by the Moment-SOS hierarchy like (\ref{eq:d-mom})-(\ref{eq:d-sos}). The
convergence property is the same. For the cleanness of the paper, we omit the details.

\subsection{Verifying solution candidates for polynomial VIPs}
\label{sc:popverify}
In this subsection,
we study how to solve the polynomial subproblem (\ref{eq:linopt}) for verifying whether a given solution candidate $u$ solves $\VIXF$ or not.
%In particular, we are mainly interested in the case that  (\ref{eq:linopt}) has at least one minimizer.
When $X$ is compact, problem (\ref{eq:linopt})  exists at least one minimizer.
Moreover, if (\ref{eq:linopt}) is unbounded from below,
then $u$ cannot be a solution.
In these cases, it suffices to choose an appropriate positive $R$ such that (\ref{eq:linopt}) with the extra boundedness constraint $\Vert x-u\Vert^2\le R$ has a negative minimum. For the rest of this subsection, we suppose (\ref{eq:linopt}) has at least one minimizer.

Under the assumption that $X$ is given by (\ref{eq:X})
with nonsingular tuple $g$, the LICQ holds for every $x\in X$.
Thus, if $v\in X$ is a minimizer of (\ref{eq:linopt}),
then it must be a KKT point of (\ref{eq:linopt}).
Introducing dual variables $\lmd(x)$, problem (\ref{eq:linopt}) can be equivalently reformulated to (\ref{eq:pop}), with
\be\label{eq:tpp3}
\left\{
\begin{aligned}
\ \theta(x)\,&\defeq\ (x-u)^T\mbF(u), \\
%%%%
\,\ \Phi(x)\,&\defeq\  \{\mbF(u) - \sum\nolimits_{i=1}^{m} \lmd_i(x) \nabla g_i(x) \}\\
&\qquad\qquad\cup\{\lmd_i(x) g_i(x)\,:\, (i\in\mc{I})\}\cup\{g_i(x)\,:\, (i\in\mc{E})\}, \\
%%%%
\,\ \Psi(x)\,&\defeq\ \{g_i(x)\, (i\in\mc{I})\}\cup\{ \lmd_i(x)\, (i\in\mc{I})\},
\end{aligned}\right.
\ee
Therefore, Algorithm~\ref{ag:momwithFT} can be applied to solve  (\ref{eq:linopt}).
Note that under our assumption,  (\ref{eq:linopt}) must be feasible,
so there is no need to check the feasibility in Step~1.
The following convergence result for Algorithm~\ref{ag:momwithFT} is straightforward by \cite[Theorems~3.3]{Nie2019} and \cite[Theorems~4.4]{Nie2018saddle}:
\begin{thm}
Suppose (\ref{eq:linopt}) has a minimizer.
Let $\theta, \Phi, \Psi$ be given as in (\ref{eq:tpp3}),
and let (\ref{eq:pop}) be the polynomial optimization reformulation for (\ref{eq:linopt}).
If either one of the following holds:
\begin{enumerate}[(i)]
    \item $\idl[\Phi]+\qmod[\Psi]$ is archemedean; or
    \item $\idl[(g_i)_{i\in\mc{E}}]+\qmod[(g_i)_{i\in\mc{I}}]$ is archemedean; or
    \item The $q(x)=0\, (q\in\Phi)$ has finitely many real solutions;
\end{enumerate}
then Algorithm~\ref{ag:momwithFT} has finite convergence.
Moreover, if we further assume that every minimizer of (\ref{eq:pop}) is an isolated critical point of (\ref{eq:linopt}),
then (\ref{eq:flatrank}) holds at each minimizer of (\ref{eq:pop}) for all $k$ that is big enough.
\end{thm}

\subsection{Solving the maximization problem (\ref{eq:max})}
\label{sc:solvewithFT}
For the maximization problem (\ref{eq:max}) in Algorithm \ref{ag:more},
we let
\be\label{eq:tpp1}
\left\{
\begin{array}{c}
\quad \theta(x)\defeq -[x]_1^T\Theta[x]_1, \quad
\Phi \defeq  E ,\\
\Psi \defeq  I \cup \{ (v-x)^T\mbF(x)\, :\, v\in V\}\cup \{ \Theta(x^*)+\delta-[x]_1^T\Theta[x]_1 \}.
\end{array}\right.
\ee
Then the maximization problem (\ref{eq:max}) can be equivalently rewritten as the polynomial minimization problem (\ref{eq:pop}),
and their $k$th order Moment-SOS relaxations are given by (\ref{eq:d-mom}-\ref{eq:d-sos}).
Algorithm~\ref{ag:momwithFT} can be applied to implement the Moment-SOS hierarchy to solve (\ref{eq:max}).

For (\ref{eq:max}), its feasible set is always nonempty since $x^*$ is a feasible point,
so there is no need to check its feasibility.
Moreover, if the optimal value of (\ref{eq:max}) is nonnegative,
then the $\delta$ which defines (\ref{eq:max}) is applicable for (\ref{eq:moreopt}).
For the $k$th ordered moment relaxation (\ref{eq:d-mom}) of (\ref{eq:max}), if its minimum value $\vartheta_k$ satisfies $\vartheta_k\ge \Theta(x^*)$,
then the optimal value of (\ref{eq:max}) is nonnegative and we may terminate Algorithm~\ref{ag:momwithFT} directly.
This is because the $\vartheta_k$ provides a lower bound for the maximum of (\ref{eq:max}).
However, if $\vartheta_k<0$, then we still need \reff{eq:flatrank} to hold for the guaranteed optimality.
Moreover, the archimedeanness is always satisfied by (\ref{eq:max}),
since it has the inequality constraint
$ [x]_1^T \Theta [x]_1 \ge [x^*]_1^T \Theta [x^*]_1 + \dt$.

The convergence of Algorithm~\ref{ag:momwithFT} is as follows.
If (\ref{eq:d-mom}) is infeasible,
then (\ref{eq:pop}) has an empty feasible set.
This is because the $[x]_d$ is feasible for (\ref{eq:d-mom}) for every feasible point $x$ of (\ref{eq:pop}).
Suppose (\ref{eq:pop}) has a nonempty feasible set and $\idl[\Phi]+\qmod[\Psi]$ is archimedean.
Then $\vartheta_k \to \vartheta_{\min}$ as $k \to \infty$ \cite{Las01}.
Furthermore, under some classical optimality conditions,
the finite convergence is guaranteed \cite{nie2014optimality}.
Lastly, if $\Phi=0$ has finitely many real solutions, then Algorithm~\ref{ag:momwithFT} has finite convergence, even if the archimedeanness is not satisfied.
Note that by Theorem~\ref{tm:solfinite}, when the $\VIXF$ is given by general polynomials,
the $\Phi=0$ has at most finitely many complex solutions.

\section{Numerical Experiments}
\label{sc:ne}
In this section, we apply our methods to solve polynomial VIPs.
We apply the software {\tt GloptiPoly~3} \cite{GloPol3}
and {\tt Mosek} \cite{mosek} to implement Moment-SOS relaxations for solving polynomial optimization problems.
The computation is implemented in a Dell OptiPlex 7090 desktop,
with an Intel\textsuperscript \textregistered \
Core(TM) i9-10900 CPU at 2.80GHz$\times $10 cores and 64GB of RAM,
in a Windows 10 operating system.
For the neatness of the paper, only four decimal digits are shown
for computational results.

\subsection{Explicit polynomial VIPs}
\label{sc:epVI}
First, we apply Algorithm~\ref{ag:KKTSDP} and Algorithm~\ref{ag:more} to solve polynomial VIPs.
In Step~3 of Algorithm~\ref{ag:KKTSDP},
if the optimal value $\varepsilon = 0$,
then $u^{(k)}$ is a solution to the $\VIXF$.
In numerical computations,
we may not have $\varepsilon= 0$ exactly due to round-off errors.
Typically,
we accept the computed candidate as a solution when $\vert \varepsilon \vert \le 10^{-6}$.
Similarly, in Step~2 of Algorithm~\ref{ag:more},
if the optimal value $\gamma^*$ satisfying $\vert\gm^*-[x^*]_1^T\Theta[x^*]_1\vert\le 10^{-6}$,
then we regard the value $\delta$ being applicable for (\ref{eq:moreopt}).
Unless specifically mentioned, we compute a complete set of solutions if it is nonempty, for every example in this subsection.

\begin{exm}\rm \label{ep:Kojima}
Consider the nonlinear complementarity problem (NCP) in \cite{Pang1993}: % given by
\[0\leq \mbF(x):=\lvt\begin{array}{c}
3x_1^2+2x_1x_2+2x_2^2+x_3+3x_4-6\\
2x_1^2+x_1+x_2^2+10x_3+2x_4-2\\
3x_1^2+x_1x_2+2x_2^2+2x_3+9x_4-9\\
x_1^2+3x_2^2+2x_3+3x_4-3
\end{array} \rvt \; \perp \; x\geq 0 .\]
This NCP is equivalent to $\vi(\re^4_+,\mbF)$,
whose feasible set $X:=\re^4_+$.
There are two solutions $(\sqrt{6}/2,0,0,1/2),~ (1,0,3,0)$ to the NCP.
Let $F_i(x)$ be the $i$th entry of $\mbF(x)$.
Then, the LMEs are
\[\lmd_i=F_i(x) \fall i=1, 2, 3, 4.\]
We ran Algorithm~\ref{ag:more} and it took around 0.55 second to find solutions $u^{(1)}, u^{(2)}$ as follows:
\begin{align*}
u^{(1)}&=(1.2247, 0.0000, 0.0000, 0.5000),\quad \varepsilon=-1.5499\cdot 10^{-8};\\
u^{(2)}&=(1.0000, 0.0000, 3.0000, 0.0000),\quad \varepsilon=-3.7262\cdot 10^{-8}.
\end{align*}
%It took around 0.55 second to find these two solutions.

We also tested solving the KKT system (\ref{eq:viKKT}) without LMEs. Treating $(\lmd_1\ddd \lmd_4)$ as additional variables, we formulated the polynomial optimization problem similar to (\ref{eq:KKTopt}) with a feasible set given by (\ref{eq:viKKT}).
Then, we found solutions to $\vi(\re^4_+,\mbF)$ under the same framework as in Algorithms~\ref{ag:KKTSDP} and \ref{ag:more}.
We presented the comparison of the time consumption for approaches with or without LMEs in Table~\ref{tab:tb1}.
In the table,
the first and the second row show the time consumption for finding the first and second solution, respectively,
and the last row is for checking the completeness of the solution set $\{u^{(1)}, u^{(2)}\}$.
It should be pointed out that the time consumption of solving (\ref{eq:max}) and (\ref{eq:moreopt}) for finding each solution or checking the completeness is included.
\begin{table}[htb]
\renewcommand{\arraystretch}{1}
% \centering
\begin{tabular}{c|c|c}  \hline
                     & \multicolumn{1}{c|}{with LMEs} & \multicolumn{1}{c}{without LMEs}\\ \hline
 the first solution  & 0.12s &  2.97s\\
 the second solution & 0.20s &  8.49s\\
 the completeness of $\{u^{(1)}, u^{(2)}\}$ & 0.23s & 10.23s\\ \hline
\end{tabular}
\caption{Computational results for Example~\ref{ep:Kojima}}
\label{tab:tb1}
\end{table}
\end{exm}

\begin{exm}\rm \label{ep:Kojima_nosol}
Consider the $\VIXF$ with
\[\mbF(x)=\lvt\begin{array}{c}
-x_1+4x_1x_2+x_2^2+x_3-x_4+1\\
2x_1^2+x_1-x_2^3-10x_3+2x_4\\
3x_1^3+x_1x_2+2x_2^2-2x_3+9x_4\\
x_1^2-3x_2^2+2x_3-3x_4-4
\end{array}\rvt,\quad X=\{x\in\re^4_+: x_1x_2x_3x_4=2\}.\]
In this example, the set $X$ is unbounded and nonconvex.
Let $F_i(x)$ be the $i$th entry of $\mbF(x)$.
Then, the LMEs are
\[\lambda_0=\frac{x^T\mbF(x)}{-8},\quad \lmd_i=F_i(x)+\lmd_0~(i=1,2,3,4).\]
We ran Algorithm~\ref{ag:KKTSDP} and the infeasibility of (\ref{eq:KKTopt}) was detected at the $3$rd ordered moment relaxation,
hence this VIP has no solution.
It took around 0.31 seconds.

Similarly to Example 5.1, we tested solving the KKT system (\ref{eq:viKKT}) without LMEs under the framework of Algorithms~\ref{ag:KKTSDP}.
For this VIP,
since we only need to check the infeasibility of (\ref{eq:KKTopt}) to detect the nonexistence of solutions,
we presented the comparison for solving moment relaxations (\ref{eq:d-mom}) with various relaxation orders $d$.
In Table~\ref{tab:tb2},
the first column is the relaxation order;
and the second and third columns represent the time consumption and whether the $d$th order moment relaxation successfully detects the nonexistence of solutions or not.
When LMEs are applied,
the $2$nd order moment relaxation is not well-defined, and the $4$th order moment relaxation is unnecessary for checking the emptiness of $\SOLXF$,
since the $3$rd order moment relaxation is already infeasible.
\begin{table}[htb]
\renewcommand{\arraystretch}{1}
% \centering
\begin{tabular}{c|c|c|c|c}  \hline
 \multirow{2}{*}{$d$}& \multicolumn{2}{c|}{with LMEs} & \multicolumn{2}{c}{without LMEs}\\ \cline{2-5}
    & time & nonexistence of solutions & time & nonexistence of solutions\\ \hline
  2 & n.a.  & n.a. & 4.52s  & not detected\\
  3 & 0.31s & detected & 32.45s & not detected\\
  4 & 1.28s & detected & 2031.27s & detected\\\hline
\end{tabular}
\caption{Computational results for Example~\ref{ep:Kojima_nosol}}
\label{tab:tb2}
\end{table}
\end{exm}

\begin{exm}\rm \label{ep:GNEP}
Consider the two-player jointly convex generalized Nash equilibrium
problem in \cite{Nie2021convex}:
\[
\mbox{1st player:}
\left\{\begin{array}{ll}
\min\limits_{x_1\in\re^3} & 10x_1^Tx_2-\sum_{j=1}^3x_{1,j}\\
\st&x_1^Tx_1+x_2^Tx_2\le1;
\end{array}\right.
\]
\[
\mbox{2nd player:}
\left\{\begin{array}{ll}
\min\limits_{x_2\in\re^3} & \sum_{j=1}^3(x_{1,j}x_{2,j})^2+(3\prod_{j=1}^3x_{1,j}-1)\sum_{j=1}^3x_{2,j}\\
\st&x_1^Tx_1+x_2^Tx_2\le1,
\end{array}\right.
\]
where $x_1:=(x_{1,1}, x_{1,2}, x_{1,3})$ and $x_2:=(x_{2,1}, x_{2,2},  x_{2,3})$ are decision variables of the first and the second player respectively.
Let $x:=(x_1,x_2)$ and denote the first (resp., second) player's objective functions by $f_1(x)$ (resp., $f_2(x)$).
Then, finding the {\it normalized equilibria} for this GNEP is equivalent to solving $\VIXF$ with
\[\mbF:=\left(\frac{\pt f_1}{\pt x_{1,1}}, \frac{\pt f_1}{\pt x_{1,2}}, \frac{\pt f_1}{\pt x_{1,3}}, \frac{\pt f_2}{\pt x_{2,1}}, \frac{\pt f_2}{\pt x_{2,2}}, \frac{\pt f_2}{\pt x_{2,3}}\right)^T,\]
\[X:=\{x\in\re^4: 1-x^Tx\ge0\}.\]
We refer to \cite{Facchinei2007book,Nie2021convex} for more details about GNEPs.
One can check that the map $\mbF$ is nonmonotone on $X$. For this polynomial VIP, the LME is
\[\lmd = -\frac{x^TF(x)}{2}.\]
We ran Algorithm~\ref{ag:more} and obtained the unique solution $u:=(u_1,u_2)$ with
\[
\begin{array}{l}
u_1=(-0.4934, -0.4934, -0.4934),\ u_2=(0.2998, 0.2998, 0.2998),
\end{array}
\]
and $\varepsilon=-2.5461\cdot 10^{-8}$.
It took around 4.21 seconds.
\end{exm}

\begin{exm}\rm \label{ep:rings}
(i) Consider the $\VIXF$ with
\be\label{eq:nonconvexFg}
\mbF(x)=\lvt\begin{array}{c}
x_1+x_2+x_3+x_4\\
x_1-x_2^2+x_3-x_4\\
-x_3-x_1x_2\\
x_4-x_1x_2
\end{array}\rvt,
\quad X=\left\{x\in\re^4:
x^Tx-1\ge0,\
2-x^Tx\ge0 \right\},
\ee
where $\mbF(x)$ is nonmonotone and $X$ is nonconvex.
For this problem, the LMEs are
\[\lmd_1=(2-x^Tx)\cdot\frac{x^TF(x)}{2},\quad
\lmd_2=(1-x^Tx)\cdot\frac{x^TF(x)}{4}.\]
We ran Algorithm~\ref{ag:more} and obtained all solutions to the $\VIXF$,
which are
\[
\begin{array}{ll}
u=(-0.2639,1.3073,-0.4537,-0.1250),&\mbox{ with } \quad \varepsilon= -2.7460\cdot10^{-9};\\
u=(0.4365,-1.0536 ,0.7694,-0.3279),&\mbox{ with } \quad \varepsilon= -7.2553\cdot10^{-9};\\
u=(-0.4108,-0.4710,1.2655,0.0899), &\mbox{ with } \quad \varepsilon= -1.1356\cdot10^{-8}; \\
u=(-0.8126,0.7417 ,0.7227,-0.5169),&\mbox{ with } \quad \varepsilon= -9.7235\cdot10^{-9}.\\
\end{array}
\]
It took around 9.56 seconds.

\noindent(ii) In (\ref{eq:nonconvexFg}),
if we change $\mbF(x)$ to
$\lvt\begin{array}{c}
-x_1-x_2-x_3-x_4\\
x_1-x_2+x_3-x_4\\
x_3-x_1x_2\\
x_4-x_1x_2
\end{array}\rvt$,
% (\ref{eq:nonconvexFg}) to $-x_1-x_2-x_3-x_4$,
then Algorithm~\ref{ag:KKTSDP} took around 1.52 seconds to detect the emptiness of the $\SOLXF$.
\end{exm}

\begin{exm}\rm \label{ep:cons_eig}
Let $A, B\in\re^{n\times n}$ be a pair of matrices such that $B$ is symmetric positive definite,
and let $C$ be a closed convex cone in $\re^n$.
Consider the reformulation $\VIXF$ for the {\it constrained eigenvalue problem} (see \cite{Facchinei2007book}) of $A$ over the cone $C$, where
\be\label{eq:cons_eig} \mbF(x):= Ax,\quad X:=\{x\in C: x^TBx=1\}.\ee
% {\rb (Are they equivalent? I guess no...)}
(i) Suppose
\be\label{eq:linearcone}
C:=\{x\in\re^n_+: x_1\ge x_2+\cdots+x_n\}.\ee
Then $X=\{x\in\re^n: g_0(x)=0,\ g_1(x)\ge0,\, \ldots,\, g_n(x)\ge0\}$, where
\[g_0(x)=x^TBx-1,\ g_1(x)= x_1- x_2-\cdots-x_n,\ g_{i}=x_i \,(i=2\ddd n). \]
The LMEs for $\VIXF$ are
\[\begin{array}{c}
\lambda_0 = \frac{x^TF(x)}{2},\ \lambda_1=F_1(x)-2\lmd_0\cdot \sum_{j=1}^n B_{1,j}x_j,\\
\lmd_i=F_i(x)-2\lmd_0\cdot \sum_{j=1}^n B_{i,j}x_j+\lmd_1\,(i=2\ddd n),\end{array}\]
where $F_i(x)$ is the $i$th entry of $\mbF(x)$ and $B_{i,j}$ means the element in the $i$th row and $j$th column of $B$.
If we let
\be\label{eq:AIn2}
A=\lvt\begin{array}{rrrr}
    -8  &  -4  &   8  &  -6\\
    -8  &  -4  &   4  &  -9\\
    -7  &  -6  &   1  &   9\\
    -6  &  -5  &  -7  &   4
\end{array}
\rvt,\qquad
B=\lvt\begin{array}{rrrr}
     4  &   0  &   3  &  -1\\
     0  &   4  &  -1  &  -2\\
     3  &  -1  &   4  &   0\\
    -1  &  -2  &   0  &   2
\end{array}
\rvt,\ee
then we obtained the unique solution $u$ to $\VIXF$ by Algorithm~\ref{ag:more} with
\[\begin{array}{c}
u=(0.5534, 0.2372, 0.0000, 0.3162),\quad \varepsilon=-1.3994\cdot10^{-8}.\\
\end{array}\]
It took around 0.50 seconds.

\noindent
(ii) Suppose $C$ is the second order cone, i.e.,
\be\label{eq:soc}
C:=\{x\in\re^n: \Vert \bar{x}\Vert_2\le x_n\},\quad \bar{x}=(x_1,x_2\ddd x_{n-1}).\ee
Then $X=\{x\in\re^n: g_0(x)=0,\ g_1(x)\ge0\}$, where $g_0(x)=x^TBx-1,\  g_1(x)=x_n^2-\bar{x}^T\bar{x}.$
We remark that in this case,
polynomial LMEs do not exist since the constraints are not nonsingular.
However, there exist rational Lagrange multiplier expressions
\[\lmd_0=\frac{x^T\mbF(x)}{2},\quad \lmd_1=\frac{F_n(x)-2\lmd_0 \sum_{i=1}^nB_{n,i}x_i}{2x_n}.\]
One may check that the denominator $x_n$ is strictly positive over $X$,
thus we may cancel it in the KKT system.
We refer to \cite{Nie2021convex} for more details about rational LMEs and how to reformulate the KKT system with them.
We ran Algorithm~\ref{ag:more} with matrices $A$ and $B$ given as in (\ref{eq:AIn2}).
It took around 4.24 seconds and find the unique solution
\[u=(0.6906,0.5866,-0.3661,0.9773), \quad \varepsilon=-5.8198\cdot 10^{-7}.\]
\end{exm}

\begin{exm}\rm
\label{ep:inv_cap}
Consider the invariant capital stock problems \cite{Jones1982}.
Let $A\in\re^{n_2\times n_1}$, $B\in\re_+^{n_2\times n_1}$,
$b\in\re^{n_2}$,
$\rho\in(0,1)$ be the discount rate,
and let $f(x)$ be a convex loss function.
For a given vector $b_0\in\re^{n_2}$,
we denote by $P(b_0)$ the problem of finding a sequence of activity levels $\{x^{(t)}\}_{t=1}^{\infty}$ to solve
\be
\label{eq:Pb0}
\left\{
\begin{array}{cl}
\min\limits_{x^{(1)},x^{(2)},x^{(3)},\ldots} &  \sum_{t=1}^{\infty}\rho^{t-1}f(x^{(t)})\\
\st & Ax^{(1)}\le b_0+b,\\
    & Ax^{(t)}\le Bx^{(t-1)}+b, \quad t=2,3,\ldots\\
    & x^{(t)}\in \re^{n_1}_+.\quad t=1,2,\ldots
\end{array}\right.
\ee
Then, the invariant capital stock problem is to find $u\in\re^n$ such that $\{x^{(t)}\}_{t=1}^{\infty}$ with every $x^{(t)}=u$ solves $P(Bu)$.
By \cite[Proposition~1.4.5]{Facchinei2007book},
if there exists $v\in\re^m$ such that $(u,v)$ solves $\mbox{VI}(\re^{n_1+n_2}_+,\mbF)$,
where
\be\label{eq:vicapital}
\mbF(x,y)=
\left[\begin{array}{c}
\nabla f(x) + (A^T-\rho B^T)y\\
b+(B-A)x
\end{array}
\right],\ee
then $u$ is an invariant capital stock.

We ran Algorithm~\ref{ag:more} with $n_1=4$, $n_2=3$, $\rho=0.7$, and
% \be\label{eq:deg4convex}
\begin{align*}
&f(x)=\sum_{1\le i\le j\le 4}x_i^2x_j^2+2x_1^3x_2+2x_1x_2^3+x_3^3+x_4^3-\sum_{i=1}^4 x_i,\\
&A=\left[\begin{array}{rrrr}
    -3  &   1   &  1   & -3 \\
     1  &  -1   &  3   &  1 \\
     1  &   0   & -3   &  2
\end{array}\right],\quad
B=\left[\begin{array}{cccc}
     5  &   4   &  1   &  1 \\
     0  &   2   &  0   &  5 \\
     2  &   5   &  4   &  4
\end{array}\right],\quad
b=\left[\begin{array}{r}
     1 \\
    -3 \\
     2\end{array}\right].\end{align*}
We obtained the unique solution
\[u=(0.1861,0.5845,0.1715,0.4868),\quad
  v=(-0.0000,0.2270,0.0000).\]
The error $\varepsilon= -7.3551\cdot 10^{-9}$.
It took around 5.54 seconds.
\end{exm}

\noindent\textbf{Comparison with {\tt the PATH Solver}.}
There exist many numerical algorithms for solving VIPs.
Since there is no monotonicity or convexity assumption in general for the polynomial VIPs in Section~\ref{sc:epVI},
we only compared Algorithm~\ref{ag:KKTSDP} with the damped Newton's method for solving the MiCP (the KKT system (\ref{eq:viKKT})) via the path search \cite{df1995,r1994},
which is implemented with the well-developed {\tt Matlab} software: {\tt the PATH Solver} \cite{PATH}.
We remark that we also tried some other numerical methods introduced in \cite{Facchinei2007book},
such as the hyperplane-projection method, interior point method, etc.
However, they usually cannot solve the VIPs in Section~\ref{sc:epVI}, especially when the problem is nonmonotone or nonconvex.
So we omit the numerical results of these methods for the cleanness of this paper.
We used the default parameter setting, which can be found in \cite[Table 6\&7]{PATH}.
Feasible initial points are applied for these methods,
that we use $(2,1,1,1)$ for Example~\ref{ep:Kojima_nosol},
the $(1,0,1,0)$ for Examples~\ref{ep:rings}(i-ii),
the $0.5e_1$ for Example~\ref{ep:cons_eig}(i),
and the zero vector for other problems.
Besides that, we use the zero vector as the initial point corresponding to Lagrange multipliers.

Numerical results are presented in Table~\ref{tab:comparison}.
The column `time' is the consumed time in seconds for solving the $\VIXF$ and verifying if a solution is obtained,
and the column `error' stands for the minimum of (\ref{eq:linopt}) at the computed solution.
Here, we only compute one solution for each problem, hence the time consumption in rows `Alg.~\ref{ag:KKTSDP}' may be much smaller than that in the last subsection, since we do not find more solutions or check the completeness of the solution sets.
\begin{table}[htb]
% \centering
\renewcommand{\arraystretch}{1.25}
\begin{tabular}{c||rlrlrlrl}
\hline
\multirow{2}{*}{Example} & \multicolumn{2}{c}{\tt{PATH}} & \multicolumn{2}{c}{Alg.~\ref{ag:KKTSDP}} \\\cline{2-5}
    & time & error & time & error  \\\hline
5.1 & 0.48 & $2\cdot 10^{-7}$ & 0.12 & $1\cdot 10^{-8}$ \\
5.2 & \multicolumn{2}{c}{fail} & \multicolumn{2}{c}{no solution} \\
5.3 & 0.79 & $6\cdot 10^{-9}$  & 1.54 & $2\cdot 10^{-8}$ \\
5.4(i) & 0.40 & 2.98 & 0.46 & $2\cdot 10^{-9}$ \\
5.4(ii) & 0.13 & 2.41 & \multicolumn{2}{c}{no solution} \\
5.5(i) & \multicolumn{2}{c}{fail} & 0.14 & $1\cdot 10^{-8}$ \\
5.5(ii) & \multicolumn{2}{c}{fail} & 1.51 & $6\cdot 10^{-7}$ \\
5.6 & 0.46 & $8\cdot 10^{-9}$ & 5.54 & $8\cdot 10^{-9}$ \\\hline
\end{tabular}
\caption{Comparison with {\tt PATH}}\label{tab:comparison}
\end{table}

The observations are summarized as follows:
\begin{itemize}
  \item The software {\tt PATH} failed to converge for Example~\ref{ep:Kojima_nosol} since the set of KKT points is empty.
  For Example~\ref{ep:cons_eig}(i), the {\tt PATH} solver fails to solve the problem, because the Jacobian is singular at the solution.
  Besides that,
  {\tt PATH} found the KKT point $(0.2301, -0.6948, 0.0715, 0.6776)$ for Example~\ref{ep:rings}(i),
  and found the KKT point $(0.5345, 0.0000, -0.8018, -0.2673)$ for Example~\ref{ep:rings}(ii).
  However, these KKT points obtained by {\tt the PATH solver} are not solutions to the $\VIXF$, since the feasible sets in these two examples are nonconvex.
  \item The damped Newton's method may be sensitive to the choice of initial points.
  For example, if we use $e_1$ as the initial point for Example~\ref{ep:GNEP},
  then {\tt the PATH Solver} do not converge.
  In contrast, Algorithm~\ref{ag:KKTSDP} can efficiently find solutions for all examples in Section~\ref{sc:epVI} if there exists any, or certify the emptiness of the $\SOLXF$.
\end{itemize}

\subsection{Randomly generated polynomial VIPs}
In this subsection, we tested our method on VIPs that are given by polynomials with randomly generated coefficients.
For all randomly generated VIPs, we only ran Algorithm~\ref{ag:KKTSDP}, i.e., the algorithm for finding one solution, and we do not check the uniqueness of the solution.

For all instances in the following,
we randomly generated $100$ polynomial VIPs and solved them by Algorithm~\ref{ag:KKTSDP}.
In the experiment,
a computed point is regarded as a solution to the VIP if the error $\vert \varepsilon \vert \le 10^{-6}$.
We allow $10$ maximum iterations (i.e., $\ell=10$) in Algorithm~\ref{ag:KKTSDP} for each VIP.
We remark that for VIPs with convex feasible sets,
Algorithm~\ref{ag:KKTSDP} guarantees to find a solution to the VIP at the initial loop (see Proposition~\ref{prop:cvx}). However, the computed solution may not satisfy $\vert \varepsilon \vert \le 10^{-6}$ due to the numerical error.
For such cases, we proceed to new loops of Algorithm~\ref{ag:KKTSDP} until a solution satisfying $\vert \varepsilon \vert \le 10^{-6}$ is computed or the maximum $\ell$ is reached.
All polynomial optimization subproblems in Algorithm~\ref{ag:KKTSDP} are solved by Algorithm~\ref{ag:momwithFT} with the relaxation order $k=d_0,d_0+1\ddd d_0+4$.
If the $(d_0+4)$th Moment-SOS relaxation cannot find a solution for the polynomial optimization,
then we stop the algorithm and report a failure of solving this polynomial VIP.

% We tested our algorithm on the following VIs:
\subsubsection{Randomly generated ball constrained VIPs} Consider the $\vi(\mathbb{B}_n,\mbF)$ with $\mathbb{B}_n:=\{x\in\re^n: x^Tx\le 1\}$.
Since $\mathbb{B}_n$ is convex and compact, the $\mbox{SOL}(\mathbb{B}_n,\mbF)\ne \emptyset$.
For the given dimension $n$ and degree $d$, we randomly generate the matrix $A\in\re^{n\times N}$ with $N=\binom{n+d}{n}$ using the {\tt Matlab} function {\tt randn}, and let
% \be\label{eq:randF}
$\mbF(x)=A\cdot [x]_d.$
% \ee
We report the numerical results in Table~\ref{tab:ballVI}.
In the table, $(n,d)$ is the pair of the dimension of $x$ and the degree of every $F_i(x)$,
`SR' (success rate) is the rate of finding a solution for the VIP successfully,
and `time' is the average consumed time (in seconds) of finding solutions.
\begin{table}[htb]
\begin{tabular}{c||ccccccccccccc}
   \hline
   $(n,d)$  & (4,2)  & (4,3) & (4,4) & (4,5) & (4,6) & (4,7) & (4,8) & (4,9) \\ \hline
   SR & 100\% & 100\% & 100\% & 100\% & 100\% & 100\% & 100\% & 100\% \\
   time & 0.25s & 0.28s & 0.34s & 0.67s & 1.03s & 2.51s & 3.95s & 8.34s  \\\hline
   $(n,d)$  & (4,10) & (5,2)  & (5,3) & (5,4) & (5,5) & (5,6) & (5,7) & (5,8) \\ \hline
   SR & 100\% & 100\% & 100\% & 99\% & 100\% & 100\% & 100\% & 100\% \\
   time & 20.49s & 0.28s & 0.43s & 1.22s & 3.82s & 9.26s & 24.63 & 70.79s  \\\hline
   $(n,d)$  & (6,2) & (6,3) & (6,4)  & (6,5) & (6,6) & (7,2) & (7,3) & (7,4) \\ \hline
   SR & 100\% & 100\% & 100\% & 99\% & 100\% & 100\% & 100\% & 100\% \\
   time & 0.92s & 1.52s & 7.43s & 35.38s & 99.51s & 1.58s & 9.79s & 41.18  \\\hline
   $(n,d)$  & (7,5) & (8,2) & (8,3) & (8,4)  & (9,2) & (9,3) & (10,2) & (11,2) \\ \hline
   SR & 100\% & 100\% & 100\% & 100\% & 100\% & 100\% & 99\% & 100\% & \\
   time & 365.49s & 4.60 & 43.98s & 213.74s & 12.44s & 188.38s & 47.22s &  121.70s\\\hline
   % $(n,d)$  & (2,2) & (2,3) & (2,4) & (3,2) \\ \hline
   % SR & 100\% & 100\% & 100\% & 100\% \\
   % time & 0.20s & 0.24s & 0.20s & 0.20 \\\hline
\end{tabular}
\caption{Computational results for randomly generated ball constrained VIPs}%\label{tab:n=4to15}
\label{tab:ballVI}
\end{table}

For this problem, we would like to remark on the following:
\begin{itemize}
\item When $(n,d)=(5,4)$, there is one case of VIP where Algorithm~\ref{ag:KKTSDP} can only find a solution whose error $\vert \varepsilon \vert =9.04\cdot 10^{-6}>10^{-6}$.
\item When $(n,d)=(6,5)$, there is one case of VIP where Algorithm~\ref{ag:KKTSDP} can only find a solution whose error $\vert \varepsilon \vert =5.78\cdot 10^{-6}>10^{-6}$.
\item When $(n,d)=(10,2)$, there is one case of VIP where we cannot find a minimizer for (\ref{eq:genopt}) due to the limit of memory.
\end{itemize}
\subsubsection{Randomly generated constrained eigenvalue problems}
Consider the constrained eigenvalue problem in Example~\ref{ep:cons_eig}.
For the given dimension $n$,
we randomly generate matrices $A\in\re^{n\times n}$ and $\hat{B}\in\re^{n\times n}$ using the {\tt Matlab} function {\tt randn}, and let $B:=\hat{B}^T\hat{B}$.
Then $x\in\re^n$ solves the constrained eigenvalue problem defined by $A$ and $B$ if it solves the $\VIXF$ given by (\ref{eq:cons_eig}).
We ran Algorithm~\ref{ag:KKTSDP} for the cone given in (\ref{eq:linearcone}) and the second order cone as in (\ref{eq:soc}),
and reported numerical results in Tables~\ref{tab:cons_eig}.
In the table, the column `$C$' is the cone $C$ we use for the constrained eigenvalue problem,
`$n$' is the dimension for $x$,
`SR' is the rate of solving the $\VIXF$ successfully, that either a solution is obtained or the nonexistence of solutions is detected,
and `time' has the same meaning as in Table \ref{tab:ballVI}.
\begin{table}[htb]
\begin{tabular}{cc||ccccccccccccc}
   \hline
   $C$ & $n$   & 3 & 4 & 5 & 6 & 7 & 8  \\ \hline
   \multirow{2}{*}{(\ref{eq:linearcone})} & SR &  100\% & 100\% & 98\% & 100\% & 100\% & 98\% \\
   & time &  0.60s & 0.81s & 0.82s & 1.46s & 8.03 & 20.85\\\hline
   \multirow{2}{*}{(\ref{eq:soc})} &   SR &  100\% & 100\% & 100\% & 99\% & 98\% \\
   & time &  0.98s & 1.31s & 2.21s & 7.04s & 45.24\\\hline
\end{tabular}
\caption{Computational results for randomly generated constrained eigenvalue problems}
\label{tab:cons_eig}
\end{table}

For the case that $C$ is given by (\ref{eq:linearcone}), we would like to remark on the following
\begin{itemize}
    \item When $n=5$, there is one case of VIPs where Algorithm~\ref{ag:KKTSDP} can only find a solution whose error $\vert \varepsilon \vert =9.50\cdot 10^{-6}>10^{-6}$.
    Moreover, there is another VIP where the Moment-SOS hierarchy cannot find a global minimizer for (\ref{eq:genopt}).
    \item When $n=7$, there are two cases of VIPs where we cannot find a minimizer for (\ref{eq:genopt}) due to the limit of memory.
    \item When $n=8$, there is one VIP where we cannot find a minimizer for (\ref{eq:genopt}) due to the limit of memory,
    and there is another VIP where we can only find a solution whose error $\vert \varepsilon \vert =5.04\cdot 10^{-6}>10^{-6}$.
\end{itemize}

For the case that $C$ is given by (\ref{eq:soc}), we would like to remark the following
\begin{itemize}
    \item When $n=6$, there is one case of VIPs where we can only find a solution whose error $\vert \varepsilon \vert =2.65\cdot 10^{-6}>10^{-6}$.
    \item When $n=7$, there are two cases of VIPs where we cannot find a minimizer for (\ref{eq:genopt}) due to the limit of memory.
\end{itemize}

\subsubsection{Randomly generated invariant capital stock problems}
Consider the invariant capital stock problems (\ref{eq:Pb0}) in Example~\ref{ep:inv_cap}.
For the given dimensions $n_1,n_2$,
we randomly generate $A\in\re^{n_2\times n_1}$, $C\in\re^{n_1\times n_1}$, $b\in\re^{n_2}$ using the {\tt Matlab} function {\tt randn}, and randomly generate $B\in\re_+^{n_2\times n_1}$ using {\tt rand},
and let $\rho=0.8$, $f(x):=[x]_1^TC^TC[x]_1$.
Then, the $\vi(\re_+^{n_1+n_2},\mbF)$ is given by (\ref{eq:vicapital}).
Note that although the feasible set of this VIP is convex, it is unbounded.
It is possible that $\mbox{SOL}(\re_+^{n_1+n_2},\mbF)=\emptyset$.
We reported the numerical results in Table~\ref{tab:inv_cap}.
In the table, $(n_1,n_2)$ is the pair of dimensions,
`SR' (success rate) and `time' have the same meaning with them in Table \ref{tab:ballVI}.
We remark that the $\VIXF$ computed in Example~\ref{ep:inv_cap} has higher degrees,
and we checked the uniqueness of the computed solution,
thus the time consumption is larger than that in Table~\ref{tab:inv_cap}.
\begin{table}[htb]
\begin{tabular}{c||ccccccccccccc}
   \hline
   % $(n,m)$  & (2,2)  & (3,2) & (4,2) & (5,2) & (6,2) & (7,2) \\ \hline
   % SR & 100\% (14) & 100\% (10) & 100\% (3) & 100\% (1) & 100\% (3) & 100\% (0) \\
   % time & 0.24s & 0.24s & 0.25s & 0.26s & 0.28s & 0.26s \\\hline
   $(n_1,n_2)$   & (4,2) & (5,2) & (6,2) & (7,2) & (8,2) & (9,2) \\ \hline
   SR &  100\%  & 100\%  & 100\%  & 100\%  & 100\%  & 99\%  \\
   time &  0.25s & 0.26s & 0.28s & 0.26s & 0.32s & 0.29s \\\hline
   $(n_1,n_2)$   & (4,3) & (5,3) & (6,3) & (7,3) & (8,3) & (9,3) \\ \hline
   SR &  100\%  & 100\%  & 100\%  & 100\%  & 100\%  & 100\%  \\
   time &  0.26s & 0.29s & 0.30s & 0.34s & 0.40s & 0.42s \\\hline
   $(n_1,n_2)$   & (4,4) & (5,4) & (6,4) & (7,4) & (8,4) & (9,4) \\ \hline
   SR &  100\%  & 100\%  & 100\%  & 100\%  & 99\%  & 100\%  \\
   time &  0.35s & 0.33s & 0.43s & 1.03s & 0.47s & 51.75s \\\hline
   % $(n_1,n_2)$   & (5,5) & (5,4) & (6,4) & (7,4) & (8,4) & (9,4) \\ \hline
   % SR &  100\% (12) & 100\% (11) & 100\% (5) & 100\% (2) & 100\% (4) & 100\% (3) \\
   % time &  0.49s & 0.33s & 0.43s & 1.03s & 14.49s & 51.75s \\\hline
\end{tabular}
\caption{Computational result for randomly generated invariant capital stock problems}
\label{tab:inv_cap}
\end{table}

For this problem, we would like to remark on the following:
\begin{itemize}
\item When $(n_1,n_2)=(9,2)$, there is a VIP where we can only find a solution whose error $\vert \varepsilon \vert =5.22\cdot 10^{-6}>10^{-6}$.
\item When $(n_1,n_2)=(8,4)$, there is a VIP where the Moment-SOS hierarchy cannot find a global minimizer for (\ref{eq:genopt}).
\item When $(n_1,n_2)=(9,4)$, there are two VIPs where Algorithm~\ref{ag:KKTSDP} took more than 2000 seconds to solve each one of them (2371.65s and 2740.15s respectively),
and the average time consumption for solving all other VIPs is 0.65s.
\end{itemize}

\begin{acknowledgement}
The research of Defeng Sun is supported in part by the NSFC/RGC Joint Research Scheme under grant  N$\_$PolyU504/19.
The research of Xindong Tang is partially supported by the Start-up Fund P0038976/BD7L from The Hong Kong Polytechnic University.
The research of Min Zhang is partially supported by the National Natural Science Foundation of China under grant 12101598.
\end{acknowledgement}

% \appendix
% \appendixpage
% \addappheadtotoc
% \appendix{
\begin{appendices}
\normalsize
\section{Generic properties of polynomial VIPs}
\label{sc:degree}
First, we review some basics of algebraic geometry.
Let $\cpx$ be the complex field and $\cpx[x]$ be the ring of polynomials
in the variable $x = ( x_{i,j} ).$
An {\it ideal} $I$ of $\cpx[x]$ is a subset of $\cpx[x]$ such that
$a+b \in I$ for all $a, b \in I$ and
$q\cdot p\in I$ for all $p\in I$ and $q\in \cpx[x]$.
Given the ideal $I$,
we say the polynomial tuple $(p_1\ddd p_m)$ is a generator for $I$ if for every $p\in I$,
we have
\[p=p_1\cdot q_1+\dots+p_m\cdot q_m \mbox{ for some }q_1\ddd q_m\in\cpx[x].\]
Every ideal is generated by finitely many polynomials.
For an ideal $I$, the set
$V(I) \, \coloneqq \,
\{x\in\cpx^n: p(x) = 0 ,\, \forall\, p \in I
\}$
is called the variety of $I$.
Such a set is called an {\it affine variety}.

Let $\td{x}:=(x_0,x_1\ddd x_n)$ be the tuple of variables in $\cpx^{n+1}$.
The $n$-dimensional complex projective space $\Pj^n$
is the set of all lines passing through the origin of $\cpx^{n+1}.$
A point $\td{x}$ in $\Pj^n$ has the coordinate
$[x_0 : x_1 : \cdots :x_n]$ such that at least one of $x_i$ is nonzero.
An ideal generated by homogeneous polynomials in $\cpx[x]$ is called a {\it homogeneous ideal}.
Given the homogeneous ideal $\td{I}$,
the set of all $\tdx\in\Pj^n$ such that $p(\tdx)=0$ for all $p\in \td{I}$ is called the {\it projective variety} of $\td{I}$.
Moreover, the {\it Zariski topology} over $\Pj^n$ (resp., $\cpx^n$) is the topology space such that all closed sets are given by projective (resp., affine) varieties, and open subsets of projective varieties are called {\it quasi-projective varieties}.
For the given quasi-projective varieties $\mc{X}$,
closed subsets are given by intersections of $\mc{X}$ with a projective variety.
It is clear that both projective and affine varieties are quasi-projective.
Besides that, we say a variety being {\it irreducible} if it is not the union of two proper closed subsets.
Any variety can be decomposite as the union of finitely many irreducible varieties, which are called {\it irreducible components}.

For the irreducible quasi-projective variety $\mc{X}$ in $\Pj^n$,
its {\it dimension} $\dim\mc{X}$ is the largest nonnegative integer $k$ such that the intersection of $\mc{X}$ and $k$ general hyperplanes is nonempty.
The dimension of a reducible variety is defined by the maximum dimensions of all its irreducible components.
The {\it codimension} $\cod(\mc{X})$ for the quasi-projective variety $\mc{X}$ in $\Pj^n$ equals $n-\dim(\mc{X})$.
It is clear that $\dim\Pj^n=n$,
and the dimension of all hypersurfaces in $\Pj^n$ equals $n-1$.
If $\mc{Y}$ is the hypersurface defined by a homogeneous polynomial $p(\tdx)$ that is not identically zero on any irreducible component of $\mc{X}$ and $\mc{X}\cap\mc{Y}\ne\emptyset$,
then \[\cod(\mc{X}\cap\mc{Y})=\cod\mc{X}+1.\]

\subsection{Proof of Theorem~\ref{tm:solfinite}}
Now, we show the finiteness of $\mc{K}$ defined in Subsection \ref{subs:KKT} and the algebraic degree of the $\VIXF$.
Recall that when some constraint qualification condition holds at $x\in\SOLXF$,
$x$ satisfies the KKT conditions (\ref{eq:viKKT}) for the $\VIXF$.
Consider the following polynomial system
\be\label{eq:cpxKKT}
\left\{
\begin{array}{c}
\mbF(x) = \sum_{i=1}^{m} \lambda_i \nabla g_i(x),\\
\lambda_1\cdot g_1(x)=\dots= \lambda_m\cdot g_m(x)=0,\ g_i(x)=0\, (i\in\mc{E}).
\end{array}
\right.
\ee
Then, the (\ref{eq:cpxKKT}) consists of all equalities in (\ref{eq:viKKT}).
We show the finiteness of complex solutions to (\ref{eq:cpxKKT}) when all $F_1\ddd F_n, g_1\ddd g_m$ are generic polynomials.
Let $\mc{A}$ be an active constraint labelling set.
Without loss of generality,
assume $\mc{A}=\{1\ddd \hat{m}\}$ with $\hat{m}\le m$.
From the generality of $g_1\ddd g_m$,
the $\hat{m}$ cannot be greater than $n$.
This is because the intersection of more than $n$ hypersurfaces in $\cpx^n$ given by general polynomials in $x$ is empty.
Moreover, if we let
% \begin{align*}
\[
J(x) \defeq \lvt\ F(x)\quad \nabla g_1(x)\quad \dots \quad \nabla g_{\hat{m}}(x)\ \rvt, \quad
\mc{V} \defeq \left\{ x\in\cpx^n: \rank{J(x)}\le \hat{m}
\right\},\]
% \end{align*}
then $\mc{V}$ is a complex variety defined by vanishing all $(\htm+1)\times(\htm+1)$ minors of $J(x)$.
% and we have $x\in\mc{V}$.
Denote \[\mc{U}\defeq\{x\in\cpx^n:g_1(x)=\dots=g_{\htm}(x)=0\}.\]
If $\mc{V}\cap\mc{U}=\emptyset$, then there does not exist $x$ satisfying (\ref{eq:cpxKKT}) with active labeling set $\mc{A}$.
\begin{proof}[Proof of Theorem~\ref{tm:solfinite}]
For any given active labeling set $\mc{A}$, when all $(g_i(x))_{i\in\mc{A}}$ are general polynomials,
the Jacobian matrix $\lvt\, \nabla g_i(x) \, \rvt_{i\in\mc{A}}$ is full column ranked over $\cpx^n$, by \cite[Proposition~2.1]{nie2009} (see also \cite[Theorem~3.2]{nie2012}).
Therefore, when $g_1\ddd g_m$ are all general, the LICQ holds at every point in $\cpx^n$,
and the $x\in\SOLXF$ only if there exist $\lmd_1\ddd \lmd_m$ such that (\ref{eq:viKKT}) holds.
Moreover, the finiteness of $\vert\SOLXF\vert $ is implied by the finiteness of solutions to (\ref{eq:cpxKKT}).

If the solution set for (\ref{eq:cpxKKT}) is finite for any given active labelling set $\mc{A}$, then (\ref{eq:cpxKKT}) has finitely many solutions by enumerating all possibilities of $\mc{A}$.
Without loss of generality, we assume $m<n$ and all constraints are active, i.e., $\mc{A}=[m]$, for convenience of our discussion.

For each $g_i(x)$, its homogenization is
\be\label{eq:homg}\td{g}_i(\tdx)\defeq x_0^{b_i}\cdot g\left(\frac{x_1}{x_0}\ddd \frac{x_n}{x_0}\right).\ee
Denote $\wtU\defeq \{\tdx\in\Pj^n: \td{g}_1(\tdx)=\dots=\td{g}_m(\tdx)=0\}.$
It is a projective variety which contains $\mc{U}$ up to the natural embedding from $\cpx^n$ to $\Pj^n$; see \cite{Shafarevich2013}.
When $g_1\ddd g_m$ are general polynomials, we have $\dim\wtU=n-m$
by Bertini's theorem \cite[Theorem~A.1]{nie2009}.
Moreover, since $g_i$ is general,
we have $\nabla_x \td{g}(\tdx)=x_0^{d_i-1}\nabla_x  g\left(\frac{x_1}{x_0}\ddd \frac{x_n}{x_0}\right)$ and the homogenization of $\nabla g_i$ equals the last $n$ components of $\nabla_x \td{g}(\tdx)$.
For each $i\in[N]$,
we let
\begin{align*}
    \td{F}_i(\tdx)&\defeq x_0^{a_1}\cdot F_i\left(\frac{x_1}{x_0}\ddd \frac{x_n}{x_0}\right),\quad \td{\mbF}(\tdx)\defeq (\td{F}_1(\tdx),\td{F}_2(\tdx)\ddd \td{F}_n(\tdx)),\\
    \wt{J}(\tdx)& \defeq \lvt\ \td{\mbF}(\tdx)\quad \nabla \td{g}_1(\tdx)\quad \dots \quad \nabla \td{g}_m(\tdx)\ \rvt.
\end{align*}
% For each $i,j$, the $(\wt{J}(x))_{i,j}$ is the homogenization of $(J(x))_{i,j}$.
Then $\wtV\defeq \left\{ \tdx\in\Pj^n: \rank\wt{J}(x)\le m \right\}$ is a projective variety that is given by homogeneous polynomials whose degrees are $c\defeq a_1+b_1+\dots+b_m-m$.
Indeed, $\mc{V}$ is a quasi-projective variety that is given by $\wtV\setminus\{x_0\ne0\}$
and $\dim\mc{V}=\dim\wtV$.
Let $\wtW\defeq \mc{V}\cap\wtU$.
Then, $\wtW$ is an open subset of $\wtU$, and the finiteness of the solution set to (\ref{eq:cpxKKT}) follows directly if $\dim\,\wtW=0$.

Note that under our assumption,
$\nabla \td{g}_1(\td{x}) \ddd \nabla \td{g}_m(\td{x})$ are linearly independent for all $\td{x}\in\widetilde{U}$,
i.e., $\wt{J}(\tdx)=m$.
For each tuple of indices $\sigma=(s_1\ddd s_m)\subset[n]$,
if we let
\[\wt{J}_{\sigma}\defeq\lvt\begin{array}{cccc}
\frac{\pt \td{g}_1}{\pt x_{s_1}} & \frac{\pt \td{g}_2}{\pt x_{s_1}} & \dots & \frac{\pt \td{g}_{m}}{\pt x_{s_1}}\\
\frac{\pt \td{g}_1}{\pt x_{s_2}} & \frac{\pt \td{g}_2}{\pt x_{s_2}} & \dots & \frac{\pt \td{g}_{m}}{\pt x_{s_2}}\\
\vdots & \vdots & \dots & \vdots \\
\frac{\pt \td{g}_1}{\pt x_{s_{m}}} & \frac{\pt \td{g}_2}{\pt x_{s_{m}}} & \dots & \frac{\pt \td{g}_{m}}{\pt x_{s_{m}}}\\
\end{array}\rvt,\quad \mc{X}_{\sigma}=\{\tdx\in \wtU: \det \wt{J}_{\sigma}\ne0\},\]
then all $\mc{X}_{\sigma}$ are quasi-projective varieties and
\be\label{eq:s-decom}\wtU=\bigcup_{\sigma} \mc{X}_{\sigma}, \quad \wtW= \bigcup_{\sigma} \mc{X}_{\sigma}\cap\mc{V}.\ee

For the given $\sigma$ and constraining polynomials $g_1\ddd g_m$,
we consider the intersection $\wtW_{\sigma}\defeq \mc{X}_{\sigma}\cap\mc{V}$.
Since $\rank\, \wt{J}_{\sigma}=m$ on $\mc{X}_{\sigma}$,
the $\wtW_{\sigma}$ is given by vanishing all $(m+1)\times(m+1)$ minors of $\widetilde{J}(\tdx)$ containing rows indexed by $\sigma$,
over the quasi-projective variety $\mc{X}_{\sigma}\setminus \{x_0=0\}$.
There are $n-m$ such minors in total,
and each one of them has exactly one row that is not indexed by entries in $\sigma$.
Denote this row by $s_0$, and the according minor by $\mathfrak{m}_{s_0}$.
Then $\mathfrak{m}_{s_0}$ is a homogeneous polynomial parameterized by the coefficients of $F_{s_0}$.
Let $\Lambda$ be a subset of $[n]\setminus\sigma$ (possibly empty).
For all $s_0\notin \sigma\cup\Lambda$, the $\mathfrak{m}_{s_0}$ being identically zero on some irreducible components of
\[\mc{L}\defeq (\bigcap_{s\in\Lambda}\{\mathfrak{m}_s=0\})\cap(\mc{X}_{\sigma}\setminus \{x_0=0\})\]
gives a proper closed subset in the space of coefficients for $F_{s_0}$,
which can be regarded as a projective space.
Therefore, by \cite[Theorem~1.22]{Shafarevich2013}, if $F_{s_0}$ is a general polynomial in $\cpx[x]_{a_{s_0}}$, then
\[\dim\,\mc{L}\cap \{\mathfrak{m}_{s_0}=0\}=\dim\,\mc{L}-1.\]
Note that $\dim\mc{X}_{\sigma}=n-m$ since $\mc{X}_{\sigma}$ is open in $\mc{U}$.
We conclude $\dim\wtW_{\sigma}=0$ when $F_1\ddd F_n$ are general polynomials by induction,
and the finiteness of $\wtW$ follows from (\ref{eq:s-decom}).
\qed\end{proof}

\subsection{The algebraic degree of polynomial VIPs}
When $\vert \mc{K} \vert $ is finite,
Algorithms~\ref{ag:KKTSDP} and \ref{ag:more} terminate within finitely many steps.
Moreover, the quantity $\vert \mc{K} \vert +1$ gives an upper bound for the number of iterations needed for these algorithms to terminate.
In general, it is difficult to characterize the number of real solutions to the KKT system.
As we mentioned before, the $\mc{K}$ is a subset of the intersection of quasi-projective varieties $\mc{U}\cap\mc{V}$,
which is zero-dimensional when all defining polynomials for $\VIXF$ are general.
Therefore, we may estimate the quantity $\vert \mc{K} \vert $ by $\vert\mc{U}\cap\mc{V}\vert$, if the latter one is also finite.

Let $\mc{X}$ be a $k$-dimensional irreducible projective variety.
The {\it algebraic degree} of $\mc{X}$, denoted by $\deg \mc{X}$, is the number of points in $\mc{X}\cap\mc{Z}$,
where the $\mc{Z}$ is a general $(n-k)$-dimensional hyperplane.
In particular, if $\dim \mc{X}=0$, the $\deg\mc{X}$ counts the number of points in $\mc{X}$ (counting multiplicities).
When $\mc{X}$ is reducible and $\dim\mc{X}=k$, the $\deg\mc{X}$ is the sum of degrees of all $k$-dimensional irreducible components for $\mc{X}$.

For two projective varieties $\mc{X}$ and $\mc{Y}$,
we say they intersect properly if $\dim \mc{X}\cap\mc{Y}=\dim\mc{X}+\dim\mc{Y}-n$.
The following result, called B\'{e}zout's Theorem \cite{Harris1992}, is useful when considering the algebraic degree for intersections.
We refer to the textbook \cite{Harris1992} for the definition of smooth varieties and transversal intersections.
\begin{thm}
\label{tm:bezout}
Let $\mc{X}$ and $\mc{Y}$ be projective varieties in $\Pj^n$ with dimensions $k$ and $l$ respectively.
If $k+l\ge n$ and $\mc{X}$ intersect $\mc{Y}$ properly,
then
$\deg \mc{X}\cap\mc{Y} \le \deg \mc{X}\cdot \deg \mc{Y}.$
In particular, the equality holds when $\mc{X}$ and $\mc{Y}$ intersect transversely.
\end{thm}

For the positive integer $r$,
denote by $S_r$ the $r$th complete symmetric function on $k$ letters $(n_1, n_2, . . . , n_k)$,
that
\[S_r\defeq \sum_{i_1+i_2+\dots+i_k=r} n_1^{i_1}n_2^{i_2}\dots n_k^{i_k}.\]
\begin{thm}  \label{tm:deg}
For the $\VIXF$, let $a_1\ddd a_n\in\N$ and $b_1\ddd b_m\in\N_+$ be degrees for $F_1\ddd F_n$ and $g_1\ddd g_m$ respectively.
If (\ref{eq:cpxKKT}) is zero-dimensional,
% i.e., there are finitely many complex KKT points,
and all constraints are active,
then the number of complex KKT points is bounded by
\be\label{eq:deg} b_1b_2\dots b_m\cdot S_{n-m}\left(\max_i\{a_i\}, b_1\ddd b_m\right).\ee
\end{thm}
\begin{proof}
Without loss of generality, assume $a_1=\max_i\{a_i\}$.
Given an $\epsilon>0$,
let $\Delta F_{i}\in(\cpx[x]_{a_1})$ be a tuple of general perturbations to $F_i(x)$ for each $i\in[n]$,
$\Delta g_{i}\in(\cpx[x]_{b_i})$ be a vector of general perturbations to $\nabla g_i(x)$ for each $i\in[m]$.
For the parameter $\epsilon$,
denote $(F_i)_{\epsilon}\defeq F_i+\epsilon\Delta F_{i}$ ,
$(\nabla g_i)_{\epsilon}\defeq \nabla g_i+\epsilon\Delta g_{i}$ for every $i\in[m]$,
and $\mbF_{\epsilon}\defeq((F_1)_{\epsilon}\ddd (F_n)_{\epsilon})$.
Let $(\td{F}_i)_{\epsilon}$ and $(\nabla \td{g}_i)_{\epsilon}$ be homogenizations of $(F_i)_{\epsilon}$ and $(\nabla g_i)_{\epsilon}$, and let
\begin{align*}
J_{\epsilon}(\tilde{x}) & \,\coloneqq\,
\bbm
\mbF_{\epsilon}(\tdx) &\ (\nabla\td{g}_1)_{\epsilon}(\tdx) &\ \dots &\ (\nabla\td{g}_m)_{\epsilon}(\tdx)
\ebm,\\
\widetilde{\mc{V}}_{\epsilon} & \,\coloneqq\, \left\{(\tdx,\epsilon)\in\Pj^n\times \cpx :
\rank\,\wt{J}_{\epsilon}(\tilde{x})\le m\right\}.
\end{align*}
The $\widetilde{\mc{V}}_{\epsilon}$ is a quasi-projective variety with a projection $\pi:\widetilde{\mc{V}}_{\epsilon}\to \cpx$ such that for all $\hat{\epsilon}\in\cpx$,
the fiber $\pi^{-1}(\hat{\epsilon})\subseteq\Pj^n$ is a projective variety.
Since the perturbations are general, there exists at most finitely many $\hat{\epsilon}$ such that $\pi^{-1}(\hat{\epsilon})\cap\mc{K}$ being $1$-dimensional,
and all other fibers intersection $\wtU$ transversely at finitely many points.
Thus $\dim (\widetilde{\mc{V}}_{\epsilon}\cap\wtU)=1$.
Furthermore, every irreducible component of $\widetilde{\mc{V}}_{\epsilon}\cap\wtU$ has a dimension equal to $1$, by \cite[Proposition~17.24]{Harris1992}.
Note that for every $u\in \mc{V}\cap\wtU=\pi^{-1}(0)\cap\wtU$,
it is contained in an irreducible component, thus there exists a curve $u(t)\subseteq \widetilde{\mc{V}}_{\epsilon}\cap\wtU$ such that the graph of $u(t)$ maps onto $\cpx$ and $u(0)=u$.
For the general $t\in\cpx$,
by genericity of every $\Delta F_i$ and $\Delta g_i$ and the elimination theory (see \cite{cox2013book}), we have $u(t)=(u_1(t)\ddd u_n(t))$ satisfies
\[
a_{\xi}(t)u_{i}(t)^{\xi}+
a_{\xi-1}(t)u_{i}(t)^{\xi-1}+\dots+
a_{1}(t)u_{i}(t)^{1}+a_0(t)=0,
\]
for every $i\in[n]$, where
$\xi=\deg (\wtV_{\epsilon}\cap \wtU)$
and all $a_{\xi}(t)$ are rational functions of the coefficients of
$F_i$, $g_i$, $\Delta F_i$ and $\Delta g_i$.
This is a univariate polynomial equation which has at most $\xi$ many solutions.
So there exist at most $\xi$ many curves $u(t)$ described in above.
Therefore, the intersection
$\wtV\cap \wtU$
has at most $\xi$ points, and $\deg (\wtV_{\epsilon}\cap \wtU_{\epsilon})$
gives an upper bound for the number of complex KKT points.

For the intersection $\deg (\wtV_{\epsilon}\cap \wtU)$,
we have
\[\deg (\wtV_{\epsilon}\cap \wtU)\le \deg \wtV_{\epsilon}\cdot \deg \wtU.\]
The $\wtU$ is given by vanishing $m$ general homogeneous polynomials $g_1\ddd g_m$,
and the degree of each $g_i$ is $b_i$.
Therefore, the algebraic degree for the complete intersection $\wtU$ is at most $b_1b_2\dots b_m$.
In $J_{\epsilon}(\tilde{x})$,
the degree for the 1st column is $a_1$, and the degree for the $i$th column equals $b_{i-1}-1$ when $i\ge2$.
By \cite[Proposition~A.6]{nie2009},
\[\dim \wtV_{\epsilon}=n-m,\quad \deg \wtV_{\epsilon}\le S_{n-m}\left(a_1, b_1\ddd b_m\right).\]
So by Theorem~\ref{tm:bezout}, we have
\[\deg (\wtV_{\epsilon}\cap \wtU)\le d_1d_2\dots d_nS_{n-m}\left(a_1, b_1\ddd b_m\right).\]
And we conclude the upper bound for $\vert \mc{K} \vert $ by noticing that $a_1=\max_i\{a_i\}$ from our assumption.
\qed\end{proof}
\begin{rmk}
For the $\VIXF$, when there exist inequality constraints,
some of them may not be active,
and we can get an upper bound for $\mc{K}$ by enumerating all possible active labeling sets.
\end{rmk}
\begin{rmk}
When $a_1=a_2=\dots=a_n$ and every $F_i(x)$ is a general polynomial in $x$ with degree equal to $a_i$,
and each $g_j(x)$ is a general polynomial in $x$ with degree equal to $a_j$,
then the upper bound (\ref{eq:deg}) for the algebraic degree is tight.
However, the upper bound (\ref{eq:deg}) is large and usually much greater than $|\mc{K}|$,
since it only provides an upper bound for all complex KKT points.
Moreover, when there exists some $j\in[n]$ such that $a_j< \max_i\{a_i\}$,
then the number of complex KKT points is smaller than the upper bound given by (\ref{eq:deg}).
\end{rmk}
\end{appendices}

\vskip 15 true pt 
\noindent {\bf {\Large Declaration}}

\vskip 10 true pt

\noindent {\bf Conflict of interest} The authors declare that they have no conflict of interest.

\end{document}